\documentclass[final]{siamltex}
\usepackage{graphicx,color}
\usepackage{amsmath,amssymb,bm}
\usepackage{latexsym}
\usepackage{mathrsfs}
\usepackage{mathtools}
\usepackage[sort&compress,square,numbers]{natbib}
\usepackage[lofdepth,lotdepth]{subfig}
\usepackage{multirow}
\usepackage{mydef}
\usepackage{algpseudocode}
\usepackage{setspace}
\usepackage[hmargin=1.25in, vmargin=1.5in]{geometry}

\allowdisplaybreaks 
\newcommand{\iii}{i}
\newcommand{\jjj}{j}
\newcommand{\mmm}{m}
\newcommand{\LLL}{L}
\newcommand{\jL}[1]{{#1},j}
\newcommand{\jcount}{j}
\newcommand{\Lmax}{L_\mathrm{max}} 
\newcommand{\LC}[1]{C_{#1}} 
\title{Accelerating stochastic collocation methods for partial differential equations with random input data
\thanks{This material is based upon work supported in part by the U.S.~Air Force of Scientific Research under grant number 1854-V521-12; by the U.S.~Department of Energy, Office of Science, Office of Advanced Scientific Computing Research, Applied Mathematics program under contract numbers ERKJ259, and ERKJE45; 
and by the Laboratory Directed Research and Development program at the Oak Ridge National Laboratory, which is operated by UT-Battelle, LLC, for the U.S.~Department of Energy under Contract DE-AC05-00OR22725.}}
\author{D.~Galindo\thanks{Joint Institute for Computational Sciences, University of Tennessee, 1 Bethel Valley Road, Oak Ridge, TN 37831 ({\tt dgalind1@utk.edu})}
\and P.~Jantsch\thanks{Department of Mathematics, University of Tennessee, Knoxville, TN 37996 ({\tt pjantsch@vols.utk.edu})}
\and C.~G.~Webster\thanks{Department of Computational and Applied Mathematics, Oak Ridge National Laboratory, 
Oak Ridge, TN 37831 ({\tt webstercg@ornl.gov}).}
\and G.~Zhang\thanks{Department of Computational and Applied Mathematics, Oak Ridge National Laboratory, 
Oak Ridge, TN 37831 ({\tt zhangg@ornl.gov}). }
     }

\begin{document}
\maketitle
\begin{abstract}\label{sec:abs}
This work proposes and analyzes a generalized acceleration technique for 
decreasing the computational complexity of using stochastic collocation (SC) methods 
to solve partial differential equations (PDEs) with random input data.  
The SC approaches considered in this effort consist of
sequentially constructed multi-dimensional Lagrange interpolant in the random parametric domain, 
formulated by collocating on a set of points so that the resulting approximation is defined
in a hierarchical sequence of polynomial spaces of increasing fidelity.
Our acceleration approach exploits the construction of the SC interpolant to accelerate the underlying ensemble of deterministic solutions.
Specifically, we predict the solution of the parametrized PDE at each collocation point on the 
current level of the SC approximation by evaluating each sample
with a previously assembled lower fidelity interpolant, 
and then use such predictions to provide deterministic (linear or nonlinear) 
iterative solvers with improved initial approximations.
As a concrete example, we develop our approach in the 
context of SC approaches that employ sparse tensor products of 
globally defined Lagrange polynomials on  
nested one-dimensional Clenshaw-Curtis abscissas.
This work also provides a rigorous computational complexity analysis of the resulting fully discrete sparse grid SC approximation, with and without acceleration, which demonstrates the effectiveness of our proposed methodology  in reducing the total number of iterations of a conjugate gradient solution of the finite element systems at each collocation point.
Numerical examples include both linear and nonlinear parametrized PDEs, which are used to illustrate the theoretical results and the improved efficiency of this technique compared with several others.
\end{abstract}



\section{Introduction}\label{sec:intro}
Modern approaches for predicting the behavior of physical and engineering problems,
and assessing risk and informing decision making in manufacturing, economic forecasting, public policy, and human
welfare, rely on mathematical modeling followed by computer simulation.  
Such predictions are obtained by constructing models whose solutions describe the phenomenon of interest, and then using computational methods to approximate the outputs of the models. 
Thus, the solution of a mathematical model can be viewed as a mapping from available input information onto a desired output of interest; predictions obtained through computational simulations are merely approximations of the images of the inputs, that is, of the output of interest.  There are several causes for possible discrepancies between observations and approximate solutions obtained via computer simulations. 
The mathematical model may not, and usually does not, provide a totally faithful description of the phenomenon being modeled.  Additionally, when an application is considered, the mathematical models need to be provided with input data, such as coefficients, forcing terms, initial and boundary conditions, geometry, etc.  This input data may be affected by a large amount of uncertainty due to intrinsic variability or the difficulty in accurately characterizing the physical system.

Such uncertainties can be included in the mathematical model by adopting a probabilistic setting, provided enough information is available for a complete statistical characterization of the physical system.  In this effort we assume our mathematical model is described by a partial differential equation (PDE) and the random input data are modeled as finite dimensional random fields, parameterized by a vector $\bm{y} = (y_1, \cdots, y_N)$ of dimension $N$, consisting of uncorrelated real-valued random variables.  Therefore, the goal of the mathematical and computational analysis becomes the approximation of the solution map ${\bm y}\mapsto u({\bm y})$, or statistical moments 
(mean, variance, covariance, etc.) of the solution or some quantity of interest (QoI) of the system, given the probability distribution of the input random data.  A major challenge associated with developing approximation techniques for such problems involves alleviating the {\em curse of dimensionality}, by which the computational complexity of any na\"{\i}ve polynomial approach will grow exponentially with the dimension $N$ of the parametric domain.  

Monte Carlo (MC) methods (see, e.g., \cite{fishman1996monte}) are the most popular approaches for approximating high-dimensional integrals, based on independent realizations $u({\bm y}_k)$, $k=1,\ldots, M$, of the parameterized PDE; approximations of the expectation or other QoIs are obtained by averaging over the corresponding realizations of that quantity. The resulting numerical error is proportional to $M^{-1/2}$, 
thus achieving convergence rates independent of dimension $N$, but requiring a very large number of samples to achieve reasonably small errors.
Other ensemble-based methods, including quasi-MC (QMC) and important sampling 
(see~\cite{Niederreiter_92,Helton_Davis_03,SW14} and the references therein), have been devised to produce increase convergence rates, e.g.,  proportional to 
$M^{-1}\log(M)^{r(N)}$, however, the function $r(N)>0$ increases with dimension $N$.
Moreover, since both MC and QMC are quadrature techniques for QoIs, 
neither have the ability to simultaneously approximate the solution map ${\bm y}\mapsto u({\bm y})$, required by a large class of applications.

In the last decade, two global polynomial approaches have been proposed that often feature much faster convergence rates: {\em intrusive} stochastic Galerkin (SG) methods, constructed from pre-defined orthogonal polynomials \cite{ghanem2003stochastic, xiu2002wiener}, or best $M$-term and quasi-optimal approaches 
\cite{Chkifa:2014kba,cohen2011analytic,devore2015quasi, BNTT14}, and {\em non-intrusive} stochastic collocation (SC) methods, constructed from (sparse) Lagrange interpolating polynomials \cite{babuvska2007stochastic, nobile2008sparse, nobile2008anisotropic}, or discrete $L^2$ projections \cite{Migliorati:2013kba,Migliorati:2014ifa}.
These methods exploit the underlying regularity of the PDE solution map $u(\bm y)$ with respect to the parameters ${\bm y}$, 
evident in a wide class of high-dimensional applications, to construct an approximate solution, and differ only in the choice of basis.

For both SG and SC approaches, the overall computational cost grows rapidly with increasing dimension. 
A recent development for alleviating such complexity and accelerating the convergence of parameterized PDE solutions is to utilize multilevel methods (see e.g., multilevel Monte Carlo (MLMC) methods  \cite{giles2008multilevel,Cliffe:2011fr,Barth:2013jz,teckentrup13,Barth:2011p5578} and the multilevel stochastic collocation (MLSC) approach \cite{teckentrup2014multilevel}). 
The main ingredient to multilevel methods is the exploitation of a hierarchical  
sequence of spatial approximations to the underlying PDE, 
which are then combined with discretizations in parameter space in such a way as to minimize the overall computational cost.  The approximation of the solution $u$ on the finest mesh is represented by the approximation on the coarsest mesh plus a sequence of ``correction'' terms. The resulting decrease in complexity with the use of multilevel methods results from the fact that the dominant behavior of the solution $u$ can be captured with cheap simulations on coarse meshes, so that the number of expensive simulations computed on fine meshes can be considerably reduced. 

Nonetheless, the dominant cost in applying any uncertainty quantification (UQ) approach lies in the solution of the underlying parametrized linear/nonlinear PDEs, for a given value of the random inputs.
Such solutions are often computed using iterative solvers, e.g., conjugate gradient (CG) methods  for symmetric  positive-definite linear systems, generalized minimal residual method (GMRES) for non-symmetric linear systems \cite{saad2003iterative}, and fixed-point iteration methods\cite{quarteroni2000numerical} for nonlinear PDEs.  
However, many high-fidelity, multi-physics models can exhaust the resources of the largest machines with a single instantiation and, as such, are not practical for even the most advanced UQ techniques. 
As such, several methods for improving the performance of iterative solvers have been proposed; especially preconditioner and subspace methods for iterative Krylov solvers.  A strategy that utilizes shared search directions for solving a collection of linear systems based on the CG method is proposed in \cite{chan1999galerkin}.  
In \cite{parks2006recycling} a technique called {\em Krylov recycling} was introduced to solve sets 
of linear systems sequentially, based on ideas adapted from restarted and truncated GMRES (see \cite{simoncini2007recent} and the references therein).  This approach was later applied to the 
linear systems that arise from SG approximations that use the so-called {\em doubly orthogonal} bases 
to solve stochastic paramterized PDEs \cite{jin2007parallel} . 
In addition, several preconditioners have been developed that improve the performance of
solving the large linear systems resulting from SG approximations that employ standard orthogonal polynomials
\cite{ghanem1996numerical, powell2009block, ernst2010stochastic,gordon2012solving}.

On the other hand, when a general linear solver is employed to solve the underling SG or SC approximation, 
it is straightforward to see that improved initial approximations can significantly reduce the number of iterations required to reach a prescribed accuracy.  A sequential orthogonal expansion is utilized in \cite{ghanem1996numerical,pellissetti2000iterative} such that a low resolution solution 
provides an initial guess for the solution of the system with an enriched basis. However, at each step, 
all the expansion coefficients must be explicitly recomputed, resulting in increased costs.
Similarly, in \cite{gordon2012solving} an extension of a mean-based preconditioner is applied to each linear system 
coming from a sequential SC approach, wherein the solution of the  $j$-th system is given as the initial vector for the $(j+1)$-th system.  This approach, as well as the Krylov recycling method, impose an ordering of the linear systems that 
appear in the SC approximation.  Consequently, new approaches are needed to amortize the cost of expensive simulations by reusing both deterministic and stochastic information across multiple ensembles of solutions.

In this work, we propose to improve the computational efficiency of 
non-intrusive approximations, by focusing on SC approaches that sequentially 
construct a multi-dimensional Lagrange interpolant in a  
hierarchical sequence of polynomial spaces of increasing fidelity.
As opposed to multilevel methods that reduce the overall computational burden by taking advantage of a
hierarchical spatial approximation, our approach exploits the structure of the SC interpolant to accelerate the underlying ensemble of deterministic solutions.
Specifically, we predict the solution of the parametrized PDE at each collocation point on the 
current level of the SC approximation by evaluating each sample
with a previously assembled lower fidelity interpolant, 
and then use such predictions to provide deterministic (linear or nonlinear) 
iterative solvers with improved initial approximations.
As a particular application, we pose this acceleration technique in the context of hierarchical SC methods  
that employ sparse tensor products of globally defined Lagrange polynomials 
\cite{nobile2008sparse, nobile2008anisotropic}, on nested one-dimensional Clenshaw-Curtis abscissas.
However,  the same idea can be extended to other non-intrusive collocation approaches 
including orthogonal polynomials \cite{xiu2002wiener}, as well as piecewise local and wavelet polynomials expansions 
\cite{BGWZ14, gunzburger2012adaptive}.

The sparse grid SC approximation considered in this work produces a sequence of interpolants, where a new set of collocation points is added on each level in order to increase the accuracy of the interpolant. For each newly added collocation point on the current level, we predict the solution of the underlying deterministic PDE using the most up to date sparse grid interpolant available; the previous level's interpolant. We then use the prediction as the starting point of the iterative solver. The uniform convergence of the sparse grid interpolant to the true solution results in an increasingly accurate initial guess as the level increases, so that the overall complexity of the SC method can be dramatically reduced.  We apply our novel approach in the context of solving both linear and nonlinear stochastic PDEs, wherein, we assume that the parameterized systems are solved by some existing linear or nonlinear iterative method. Furthermore, in the linear case, this technique can also be used to efficiently generate improved preconditioners for linear systems associated to the collocation points on higher levels, which further accelerates the convergence rate of the underlying solver.

The outline of this paper is as follows:  
We begin by describing the class of parameterized linear and nonlinear stochastic PDEs
under consideration in \S \ref{sec:setup}. 
In \S\ref{sec:algo} we describe our acceleration technique in the context of general stochastic collocation methods, 
defined on a hierarchical sequence of polynomial spaces, for approximating both linear and nonlinear stochastic elliptic PDEs using 
nonlinear iterative solvers.
In \S\ref{sec:scm} we briefly recall the sparse grid SC method, where the sparse grid interpolant is constructed with the use of nested one-dimensional Clenshaw-Curtis abscissas.  The theoretical convergence rates, with respect to the level of the interpolant and the degrees of freedom are shown in \S\ref{sec:estimate}.  In \S\ref{sec:complexity} we provide a rigorous computational complexity analysis of the resulting fully discrete sparse grid SC approximation, with and without acceleration, used to demonstrate the effectiveness of our proposed methodology in reducing the total number of iterations of a conjugate gradient solution of the finite element systems at each collocation point.
Finally, in \S\ref{sec:numerical} we provide several numerical examples, including  both moderately large-dimensional linear and nonlinear parametrized PDEs, which are used to illustrate the theoretical results and the improved efficiency of this technique compared with several others.

\section{Problem setting}\label{sec:setup}
Let $D\subset \R^d, d=1,2,3$, be a bounded 
domain and let $(\Omega, \mcF,\mathbb{ P})$ denote a complete probability space with sample space $\Omega$, $\sigma$-algebra $\mcF = 2^{\Omega}$, and probability measure $\mathbb{P}: \mathcal{F} \rightarrow [0,1]$. Define $\mathcal{L}$ as a differential operator that depends on a coefficient $a(x,\omega)$ with $x\in D$ and $\omega \in \Omega$. Analogously, the forcing term $f= f(x, \omega)$ can be assumed to be a random field as well. In general, $a$ and $f$ belong to different probability spaces but, for economy of notation, we simply denote the stochastic dependences in the same probability space. Consider the stochastic boundary value problem. Find a random function $u: \overline{D} \times \Omega\rightarrow \mathbb{R}$ such that, $\mathbb{P}$-a.e. in $\Omega$, the following equations hold:
\begin{equation}\label{Lauf}
\left\{
\begin{aligned}
\mathcal{L}(a)(u) & = f \;\mbox{  in  } D,\\
u &=\,g \; \mbox{  on  } \partial D,
\end{aligned}\right.
\end{equation}
where $g$ is a suitable boundary condition. We denote by $W(D)$ a Banach space and assume the underlying random input data are chosen so that the corresponding stochastic system \eqref{Lauf} is well-posed and has a unique solution $u(x, \omega) \in L_{\mathbb{P}}^q(\Omega; W(D))$, the function space given by
\begin{equation*}
\begin{aligned}
 L_{\mathbb{P}}^q(\Omega; W(D))
  :=\bigg\{ &u:   \overline{D} \times \Omega \rightarrow \mathbb{R} \; \Big|\;  u \mbox{ is strongly measurable and }  \int_{\Omega}\|u\|^q_{W(D)} \,d\mathbb{P}(\omega) < + \infty \bigg\}.
\end{aligned}
\end{equation*}
In this setting, the approximation space consists of Banach-space valued functions that have finite $q$-th order moments. Two example problems posed in this setting are given as follows.
\begin{example}\label{ex:poisson} {\em (Linear elliptic problem).} Find a random field $u: \overline{D}\times \Omega \rightarrow \mathbb{R}$ such that $\mathbb{P}$-a.e.
\begin{equation}\label{eq:ellip}
\left\{
\begin{array}{rll}
- \nabla \cdot (a(x, \omega)\nabla u(x, \omega))
& =\, f(x, \omega)
&\mbox{  in  } D \times \Omega, \\
u(x, \omega) 
&= \,0
&\mbox{  on  }  \partial D \times \Omega,\\
\end{array}\right.
\end{equation}
where $\nabla$ denotes the gradient operator with respect to the spatial variable $x \in D$. The well-posedness of \eqref{eq:ellip} is guaranteed in $L_{\mathbb{P}}^2(\Omega; H_0^1(D))$ with $a(x, \omega)$ uniformly elliptic, i.e.,
\begin{equation}\label{eq:uniformellipticity}
   \mathbb{P}\big(\omega \in \Omega: a_{\min} \le a(x,\omega) \le a_{\max}\
  \forall\, x\in \overline{D}\big) = 1 \,\,\,\mbox{with}\,\,\, a_{\min}, a_{\max} \in (0,\infty),
\end{equation}
and $f(x,\omega)$ square integrable, i.e.,
\begin{equation*}
   \int_{D} \mathbb{E}[f^2] dx :=
   \int_{D} \int_\Omega f^2(x,\omega) \,d\mathbb{P}(\omega)dx < +\infty.
\end{equation*}
\end{example}

\begin{example}\label{ex:nonlinear} {\em (Nonlinear elliptic problem)}.
For $k\in\mathbb{N}$, find a random field $u: \overline{D}\times \Omega \rightarrow \mathbb{R}$ such that $\mathbb{P}$-a.e.
\begin{equation}\label{eq:NLellip}
\left\{
\begin{array}{rll}
-\nabla\cdot(a(x, \omega)\nabla u(x, \omega))+ u(x, \omega)|u(x, \omega)|^{k}
&= \,f(x, \omega)
& \textrm{  in  }  D,\\
u(x, \omega)
& = \,0  
&\textrm{  on  } \partial D.
\end{array}
\right.
\end{equation}
The well-posedness of \eqref{eq:NLellip} is guaranteed in 
$L_{\mathbb{P}}^2\left(\Omega; W(D)\right)$ with $a,f$ as in Example \ref{ex:poisson} and $W(D) = H_0^1(D)\cap L^{k+2}\left(D\right)\,$\cite{nobile2008sparse}.
\end{example}

In many applications, the source of randomness can be approximated with only a finite number of uncorrelated, or even independent, random variables.  For instance, the random input data $a$ and $f$ in \eqref{Lauf} may have a piecewise representation, or in other applications may have spatial variation that can be modeled as a correlated random field, making them amenable to approximation by a Karhunen-Lo\`{e}ve (KL) expansion \cite{loeve1977probability}. In practice, one has to truncate such expansions according to 
the desired accuracy of the simulation. 
As such, we make the following assumption regarding the random input data $a$ and $f$ (cf \cite{gunzburger2014stochastic,nobile2008sparse}). 

\begin{assumption}\label{assumption1} {\em (Independence and finite dimensional noise).}
The random fields $a(x, \omega)$ and $f(x, \omega)$ have the form:
\[
a(x, \omega) = a(x,\bm y(\omega)) \;\mbox{ and }\; f(x, \omega) = f(x, \bm y(\omega)) \mbox{ on }\; D \times \Omega, 
\]
where $\bm{y}(\omega) = [ y_1(\omega), \ldots, y_N(\omega) ]: \Omega \rightarrow \mathbb{R}^N$ is a vector of independent and uncorrelated real-valued random variables. 
\end{assumption}

 We note that Assumption \ref{assumption1} and the Doob-Dynkin lemma \cite{oksendal2003stochastic} guarantee that $a(x, \bm y(\omega))$ and $f(x, \bm y(\omega))$ are Borel-measurable functions of the random vector $\bm y:\Omega\rightarrow\mathbb{R}^N$.
In our setting, we denote by $\Gamma_n = y_n(\Omega) \subset \mathbb{R}$ the image of the random variable $y_n$, and set $\Gamma = \prod_{n=1}^N \Gamma_n$, where $N \in \mathbb{N}_{+}$. If the distribution measure of $\bm{y}(\omega)$ is absolutely continuous with respect to Lebesgue measure, then there exists a joint probability density function of $\bm y(\omega)$ denoted by
\s
  \varrho(\bm{y}): \Gamma \rightarrow \mathbb{R}_{+}, \quad \text{with}\quad  \varrho(\bm{y}) = \prod_{n=1}^N \varrho_n(y_n) \in L^{\infty}(\Gamma).
\f
Therefore, based on Assumption \ref{assumption1}, the probability space $(\Omega, \mathcal{F}, \mathbb{P})$ is mapped to $(\Gamma, \mathcal{B}(\Gamma), \varrho(\bm y)d\bm y)$, where $\mathcal{B}(\Gamma)$ is the Borel $\sigma$-algebra on $\Gamma$ and $\varrho(\bm y)d\bm y$ is a probability measure on $\mathcal{B}(\Gamma)$. By assuming the solution $u$ of \eqref{Lauf} is $\sigma$-measurable with respect to $a$ and $f$, the Doob-Dynkin lemma guarantees that $u(x,\omega)$ can also be characterized by the same random vector $\bm{y}$, i.e.,
\begin{equation*}
u(x,\omega) = u(x, y_1(\omega), \ldots, y_N(\omega)) \in
L_{\varrho}^q(\Gamma; W(D)),
\end{equation*}
where $L_{\varrho}^q(\Gamma; W(D))$ is defined by
\begin{equation*}
\begin{aligned}
L^q_\varrho(\Gamma;W(D)) = \bigg\{u & :\overline{D}\times \Gamma \rightarrow \R \;\Big|\; u \mbox {  strongly measurable and } \int_\Gamma \norm{u}^q_{W(D)}\varrho(\bm y) d\bm{y} < \infty \bigg\}.
\end{aligned}
\end{equation*}
Note that the above integral will be replaced by the essential supremum when $q=\infty$: 
\begin{equation*}
\begin{aligned}
L^\infty(\Gamma;W(D)) = \bigg\{u & :\overline{D}\times \Gamma \rightarrow \R \;\Big|\; u \mbox {  strongly measurable and } \operatorname{ess\thinspace sup}_{\bm y} \norm{u(\bm y)}_{W(D)} < \infty \bigg\}.
\end{aligned}
\end{equation*}

\subsection{Weak formulation}
In what follows, we treat the solution to \eqref{Lauf} as a parameterized function $u(x,\bm y)$ of the $N$-dimensional random variables $\bm y\in \Gamma\subset\mathbb{R}^N$. This leads to a Galerkin weak formulation 
\cite{gunzburger2014stochastic} 
of the PDE in \eqref{Lauf}, with respect to both physical and parameter space, i.e., seek $u \in L_{\varrho}^q(\Gamma; W(D))$ such that
\begin{equation*}
\int_{\Gamma}\int_{D} \left( \sum_{\nu\in \Lambda_1 \cup \Lambda_2} S_{\nu}(u;\bm y)\, T_\nu(v) \right) \varrho\, dx d\bm y =  \int_\Gamma \int_D f\, v \varrho\, dxd\bm y, \;\; \forall v \in L^q_{\varrho}(\Gamma; W(D)),
\end{equation*}
where $T_\nu, \nu\in \Lambda_1\cup \Lambda_2$ are linear operators independent of $\bm y$, while the operators $S_\nu$ are linear for $\nu\in \Lambda_1$, and nonlinear for  $\nu\in \Lambda_2$. Moreover, since the solution $u$ can be viewed as a mapping $u: \Gamma \rightarrow W(D)$, for convenience we may omit the dependence on $x\in D$ and write $u(\bm y)$ to emphasize the dependence of $u$ on $\bm y$. As such, we may also write the problem \eqref{Lauf} 
in the alternative weak form
\begin{equation}\label{eq:weakform2}
\int_{D}  \left( \sum_{\nu\in \Lambda_1\cup \Lambda_2} S_\nu(u(\bm y);\bm y)\, T_\nu(v) \right) \, dx =  \int_D f(\bm y)\, v \, dx, \;\; \forall v \in W(D),\; \varrho\mbox{-a.e. in } \Gamma.
\end{equation}
Therefore, the stochastic boundary-value problem \eqref{Lauf} has been converted into a deterministic parametric problem \eqref{eq:weakform2}. The acceleration technique proposed in \S \ref{sec:algo} and the sparse-grid SC method discussed in \S \ref{sec:scm} will be based on the solution of the weak form \eqref{eq:weakform2} above.

\section{Accelerating stochastic collocation methods}\label{sec:algo}
Our acceleration scheme will be proposed in the context of both linear and nonlinear elliptic PDEs. 
A general SC approach requires the semi-discrete solution $u_h(\cdot, \bm y) \in W_h(D) \subset W(D)$ at a set of collocation points $\{ \bm y_{\jL{\LLL}} \}_{\jcount=1}^{M_\LLL}\subset\Gamma$, given by
\begin{equation}\label{eq:FEM}
u_h(x, \bm y_{\jL{\LLL}}) = \sum_{\iii=1}^{M_h} c_{\jL{\LLL},\iii}\, \varphi_\iii(x),\quad \jcount=1,\ldots,M_\LLL.
\end{equation}
Here $\{\varphi_\iii\}_{\iii=1}^{M_h}$ is a predefined finite element basis of $W_h(D)$, and for $\jcount=1,\ldots,M_\LLL$,
the coefficient vector $\bm c_{\jL{\LLL}} := (c_{\jL{\LLL},1},\ldots, c_{\jL{\LLL},M_h} )^{\top}$ is the solution of the following system of equations:
\begin{eqnarray}\label{samplelinsys}
	\lefteqn{ \sum_{\iii=1}^{M_h} c_{\jL{\LLL},\iii} \int_{D} \sum_{\nu\in \Lambda_1}S_\nu\left(\varphi_\iii;\bm y_{\jL{\LLL}}\right)\, T_\nu(\varphi_{i'}) \, dx } \\
	 & \qquad=  \int_D f(\bm y_{\jL{\LLL}}) \varphi_{i'} - \sum_{\nu\in \Lambda_2} S_\nu\left(\sum_{\iii=1}^{M_h} c_{\jL{\LLL},\iii}\,\varphi_\iii; \bm y_{\jL{\LLL}}\right) T_\nu(\varphi_{i'}) \,  \, dx, \;\; i'=1,\ldots,M_h,\notag
\end{eqnarray} with $S_\nu$ and $T_\nu$ defined as above. Note that \eqref{samplelinsys} is equivalent to \eqref{eq:weakform2} with the nonlinear operators subtracted on the right hand side. When $\Lambda_2 = \emptyset$, the PDE is linear, and a standard FEM discretization leads to a linear system of equations.

For $\LLL\in\mathbb{N}_+$, we denote by $\mathcal{I}_\LLL$ an interpolation operator that utilizes $M_\LLL$ collocation points, defined by $\mathcal{H}_\LLL=\{\bm y_{\jL{\LLL}}\}_{\jcount=1}^{M_\LLL}$.
More generally, assume that we have a family of interpolation operators $\{ \mathcal{I}_\LLL \}_{\LLL\in \mathbb{N}_+}$, which for each $\LLL\in\mathbb{N}_+$ approximates the solution $u_h(x, \cdot)$ in polynomial spaces 
\s
\mathcal{P}_1(\Gamma) \subset \ldots \subset \mathcal{P}_\LLL(\Gamma)\subset \mathcal{P}_{\LLL+1}(\Gamma)\subset \ldots \subset L^2_\varrho(\Gamma),
\f
of increasing fidelity, defined on sets of sample points $ \mathcal{H}_\LLL \subset\Gamma$.
Assume further that the fully discrete solution $u_{h,\LLL} \in W_h(D)\otimes \mathcal{P}_\LLL(\Gamma)$ has Lagrange interpolating form
\begin{equation}\label{eq:general_interpolant}
	u_{h,\LLL}(x,\bm y) := \mathcal{I}_\LLL[u_h] (x,\bm y) = \sum_{\jcount=1}^{M_\LLL} \left( \sum_{i=1}^{M_h} c_{\jL{\LLL},\iii} \varphi_i(x) \right) \Psi_{\jL{\LLL}}(\bm y),
\end{equation}
where $\{ \Psi_{\jL{\LLL}} \}_{\jcount=1}^{M_\LLL}$ is a basis for $\mathcal{P}_\LLL(\Gamma)$. The approximation \eqref{eq:general_interpolant} can be constructed by solving for $u_h(x,\bm y_{\jL{\LLL}})$ {\em independently} at each sample point $\bm y_{\jL{\LLL}}\in\mathcal{H}_{\LLL}$. 
In \S \ref{sec:scm}, we construct a specific example of an interpolation scheme satisfying \eqref{eq:general_interpolant}, namely global sparse grid collocation.

For each $\LLL\in\mathbb{N}$, the bulk of the computational cost in using \eqref{eq:general_interpolant} goes into solving the 
$M_\LLL$ systems of equations \eqref{samplelinsys} corresponding to each collocation point $\bm y_{\jL{\LLL}}, \, \jcount=1,\ldots,M_\LLL$. Since the systems are independent and deterministic, they can be solved separately using existing FEM solvers, providing a straightforward path to parallelization compared to intrusive methods such as stochastic Galerkin methods. In this work, we consider iterative solvers for the system in \eqref{samplelinsys}, and propose an acceleration scheme to reduce the total number of iterations necessary to the collection of systems over the set of sample parameters.

Denoting by $\widetilde{u}_h$ the output of the selected iterative solver for the system \eqref{samplelinsys}, for $\bm y_{\jL{\LLL}} \in \mathcal{H}_\LLL$ the semi-discrete solution $u_h(x, \bm y_{\jL{\LLL}})$ is approximated by
\[
u_h(x, \bm y_{\jL{\LLL}}) = \sum_{\iii=1}^{M_h}c_{\jL{\LLL},\iii}\, \varphi_\iii(x) \approx \widetilde{u}_h(x, \bm y_{\jL{\LLL}}) = \sum_{\iii=1}^{M_h}\widetilde{c}_{\jL{\LLL},\iii}\, \varphi_\iii(x),
\]
where we define $\widetilde{\bm c}_{\jL{\LLL}} = (\widetilde{c}_{\jL{\LLL},1}, \ldots, \widetilde{c}_{\jL{\LLL},M_h})^{\top}$, and therefore the final SC approximation is given by a perturbation of \eqref{eq:general_interpolant}, i.e.,
\begin{equation}\label{eq:solution}
	\widetilde{u}_{h,\LLL}(x, \bm y) := \sum_{\jcount=1}^{M_\LLL} \left( \sum_{\iii=1}^{M_h} \widetilde{c}_{\jL{\LLL},\iii} \,\varphi_\iii(x) \right) \Psi_{\jL{\LLL}}(\bm y).
\end{equation}

We observe that the performance of the underlying iterative solver can be improved by proposing a good initial guess, denoted $\bm c_{\jL{\LLL}}^{(0)}$, or constructing an effective preconditioner to reduce the condition number of the system. 
Here, we propose our approach for
improving initial deterministic approximations, remarking that the same idea can be also utilized to construct preconditioners. 
 To start the iterative solver for the system in \eqref{samplelinsys}, it is common to use a zero initial guess, i.e.,~$\bm c_{\jL{\LLL}}^{(0)} = (0,\ldots,0)^{\top}$. 
However, we can predict the solution at level $\LLL$ using lower level approximations to construct improved initial solutions $\bm c_{\jL{\LLL}}^{(0)}$. Assume that we first obtain $\widetilde{u}_{h,\LLL-1}(x, \bm y)$ by collocating solutions to \eqref{samplelinsys} over $\mathcal{H}_{\LLL-1}$. Then at level $\LLL$, for each new point $\bm y_{\jL{\LLL}} \in \mcH_\LLL \setminus \mcH_{L-1}$, the initial guess $\bm c_{\jL{\LLL}}^{(0)}$ can be given by interpolating the solutions from level $\LLL-1$, i.e.,
\begin{equation}\label{c0jcwyj}
 \bm c_{\jL{\LLL}}^{(0)} = \Big(\widetilde{u}_{h,\LLL-1}(x_1, \bm y_{\jL{\LLL}}), \ldots, \widetilde{u}_{h,\LLL-1}(x_{M_h}, \bm y_{\jL{\LLL}})\Big)^{\top}
 = \sum_{\jcount'=1}^{M_{\LLL-1}} \widetilde{\bm c}_{\jL{\LLL-1}'} \Psi_{\jL{\LLL-1}'}(\bm y_{\jL{\LLL}}).
\end{equation}
For a convergent interpolation scheme, we expect the necessary number of iterations to compute $\widetilde{\bm c}_{\jL{\LLL}}$ to become smaller as the level $\LLL$ increases to an overall maximum level, denoted $\Lmax$. As such, the construction of the desired solution $\widetilde{u}_{h,\Lmax}$ is accelerated through the intermediate solutions $\{\widetilde{u}_{h,\LLL}\}_{\LLL=1}^{\Lmax-1}$. 
Note that this approach reduces computational cost by improving initial guesses, but does not depend on the specific solver used. Thus, our scheme may be combined with other techniques for accelerating convergence, such as faster nonlinear solvers or better preconditioners. 
When the underlying PDE is nonlinear with respect to $u$, iterative solvers are commonly used for the solution of \eqref{samplelinsys}.
In Algorithm 1, we outline the acceleration procedure described above, using a general nonlinear iterative method for the solution of \eqref{samplelinsys}. 
\begin{table}[h!]
\begin{tabular}{p{0.95\textwidth}}
\hline\noalign{\smallskip}
{\bf Algorithm 1}: {\em The accelerated SC algorithm}\\
\hspace{.65cm}{\bf Goal:} Compute $\widetilde{u}_{h,\Lmax}(x, \bm y) := \sum_{\jcount=1}^{M_{\Lmax}} \left( \sum_{\iii=1}^{M_h} \widetilde{c}_{\jL{\Lmax},\iii} \,\varphi_\iii(x) \right) \Psi_{\jL{\Lmax}}(\bm y)$\\ [5.5pt]
\hline 
\vspace{-0.3cm}
\begin{spacing}{1.25}
\begin{algorithmic}[1]\label{algo}
\vspace{.1cm}
\State Define $M_0=1$ and $\widetilde{\bm c}_{0,1} = (0,\ldots,0)^{\top}$
\For {$\LLL = 1 ,\ldots, \Lmax$}
\For { $\bm y_{\jL{\LLL}} \in \mcH_\LLL \setminus \left(\bigcup_{l=1}^{\LLL-1}\mcH_{l}\right)$}
	\State Compute the initial guess according to \eqref{c0jcwyj}:
	\State $\bm c^{(0)}_{\jL{\LLL}} = \sum_{\jcount'=1}^{M_{\LLL-1}} \widetilde{\bm c}_{\jL{\LLL-1}'} \Psi_{\jL{\LLL-1}'}(\bm y_{\jL{\LLL}})$
	\State Initialize: $k = 1$
	\Repeat
		\State Compute residual $\bm r^{(k)}_{\jL{\LLL}} = (r_{\jL{\LLL},1}^{(k)}, \ldots, r_{\jL{\LLL},M_h}^{(k)})^{\top}$:
		\For{ $i=1,\ldots, M_h $}\vspace{.1cm}
		\State $ r^{(k)}_{\jL{\LLL},i} =  \int_D f\left(\bm y_{\jL{\LLL}}\right) \varphi_i\ - {\displaystyle \sum_{\nu\in \Lambda_1\cup\Lambda_2}} S_\nu\left(\sum_{i'=1}^{M_h} c_{\jL{\LLL},i'}^{(k)}\, \varphi_{i'}(x), \bm y_{\jL{\LLL}}\right) T_\nu(\varphi_i) \; dx$
		\EndFor
		\State Update the solution: $\bm c^{(k+1)}_{\jL{\LLL}} = \bm c^{(k)}_{\jL{\LLL}} + \mathscr{S}(\bm r^{(1)}_{\jL{\LLL}},\ldots,\bm r^{(k)}_{\jL{\LLL}})$
		\State $k= k+1$
	\Until{$\| \bm c^{(k)}_{\jL{\LLL}} - \bm c^{(k-1)}_{\jL{\LLL}} \| < \tau$ }
	\State $\widetilde{\bm c}_{\jL{\LLL}} = {\bm c}^{(k)}_{\jL{\LLL}}$ 
\EndFor
\EndFor
\end{algorithmic}
\vspace{-0.9cm}
\end{spacing}
\\
\hline
\end{tabular}
\end{table}

The efficiency of the proposed algorithm will depend crucially on the number of times the iterative solver is utilized, i.e., how many sample points are in the set $\Delta\mathcal{H}_{\LLL} = \mcH_\LLL \setminus \left(\bigcup_{l=1}^{\LLL-1}\mcH_{l}\right)$ for each level $\LLL$. In fact, if the sample points are not nested, it could be the case that $\Delta\mathcal{H}_{\LLL} = \mathcal{H}_\LLL$, and the algorithm may be very inefficient. Hence, in the following sections we will assume:
\begin{assumption}\label{assumption2} 
Assume that the multidimensional point sets $\mathcal{H}_\LLL, \LLL=1,\ldots,\Lmax$ are nested, i.e.,
\begin{equation*}
	\mathcal{H}_1 \subset \mathcal{H}_2 \subset \ldots \subset \mathcal{H}_{\Lmax}\subset\Gamma.
\end{equation*} 
Then $\Delta\mathcal{H}_{\LLL} = \mathcal{H}_{\LLL} \setminus \mathcal{H}_{\LLL-1}$, and we can construct the intermediate solutions $\{ \widetilde{u}_{h,\LLL}\}_{\LLL=1}^{\Lmax-1}$ using a subset of the information needed to approximate $\widetilde{u}_{h,\Lmax}$.
\end{assumption}
   
In \S\ref{sec:scm} we construct an interpolant using a point set which satisfies Assumption \ref{assumption2}. Next, we give several examples of Algorithm 1, using iterative solvers for both nonlinear and linear elliptic PDEs. 
\begin{example}\label{ex4}
Consider the weak form of the nonlinear elliptic PDE in Example \ref{ex:nonlinear}, letting $S_1(v;\bm y) = a(x, \bm y)\nabla v $, $T_1(v) = \nabla v$, $S_2(v, \bm y) = v(x, \bm y) |v(x, \bm y)|^s$, and $T_2(v) = v$ (note that $\Lambda_1=\{1\}$, $\Lambda_2=\{2\}$). When using the fixed point iterative method in Algorithm 1, for the update step we define
\begin{equation*}
	\mathscr{S}(\bm r^{(1)}_{\jL{\LLL}},\ldots,\bm r^{(k)}_{\jL{\LLL}}) = \bm A_{\jL{\LLL}}^{-1}\bm r^{(k)}_{\jL{\LLL}},
\end{equation*}
 where the matrix $\bm A_{\jL{\LLL}} = \bm A( \bm y_{\jL{\LLL}} ),~ \jcount=1,\ldots, M_\LLL$ is defined by
\begin{equation}\label{BJ}
[\bm A_{\jL{\LLL}}]_{i,i'} = \int_D a(\bm y_{\jL{\LLL}}) \nabla \varphi_{i'} \nabla \varphi_i \, dx, \; \mbox{ for }\, i,i' = 1, \ldots, M_h.
\end{equation}
With $u_{h,\LLL}^{(k)}(x, \bm y_{\jL{\LLL}}) = \sum_{i=1}^{M_h} c_{\jL{\LLL},\iii}^{(k)}\, \varphi_i(x)$, this update is equivalent to solving the following linear system 
\[
\int_D a(\bm y_{\jL{\LLL}}) \nabla u_{h,\LLL}^{(k+1)}\; \nabla v\; dx = \int_D \Big[ f(\bm y_{\jL{\LLL}}) - u^{(k)}_{h,\LLL}(\bm y_{\jL{\LLL}}) |u^{(k)}_{h,\LLL}(\bm y_{\jL{\LLL}})|^s \Big] \,v \,dx \;\; \forall v \in W_h(D),
\]
to update $u_h^{(k)}$ to $u_h^{(k+1)}$ at the $(k+1)$-th iteration. Note that each iteration of the solver in Algorithm 1 requires the solution of this linear system, which is not accelerated by our algorithm.
\end{example}

\begin{example}\label{ex5}
As a special case of the example above, consider the weak form of the linear elliptic problem in Example \ref{ex:poisson} with $\Lambda_1=\{1\}$, $\Lambda_2=\emptyset$, $S_1(v;\bm y) = a \nabla v$ and $T_1(v) = \nabla v$ in \eqref{samplelinsys}. Due to the linearity, at each collocation point the solution
$u_h(x, \bm y_{\jL{\LLL}}) = \sum_{i = 1}^{M_h} c_{\jL{\LLL},\iii} \varphi_i(x)$ can be approximated by solving the following linear system 
\begin{equation}\label{jlinsys}
	 \bm A_{\jL{\LLL}} \bm c_{\jL{\LLL}}= \bm f_{\jL{\LLL}},
\end{equation}
with $\bm A_{\jL{\LLL}} = \bm A(\bm y_{\jL{\LLL}}), \, \jcount=1,\ldots,M_\LLL$ as in \eqref{BJ}, and $ (\bm f_{\jL{\LLL}})_i = \int_D f(x, \bm y_{\jL{\LLL}}) \varphi_i(x)dx$ for 
$\iii= 1, \ldots, M_h$. 
Under our assumptions on the coefficient $a$, the linear system \eqref{jlinsys} is symmetric positive definite, and we can use the CG method 
\cite{saad2003iterative} 
to find its solution. For $k\in\mathbb{N}^{+}$, by recursively defining 
\begin{equation*}
\bm p_{\jL{\LLL}}^{(k)} = \bm r_{\jL{\LLL}}^{(k)} - \sum_{k'<k} \frac{\bm p_{\jL{\LLL}}^{(k')\top}\bm A_{\jL{\LLL}} \bm r_{\jL{\LLL}}^{(k)}}{\bm p_{\jL{\LLL}}^{(k')\top} \bm A_{\jL{\LLL}} \bm p_{\jL{\LLL}}^{(k')}}\bm p_{\jL{\LLL}}^{(k')},
\end{equation*}
we get the update function
\begin{equation*}
	\mathscr{S}(\bm r^{(1)}_{\jL{\LLL}},\ldots,\bm r^{(k)}_{\jL{\LLL}}) = \frac{\bm p_{\jL{\LLL}}^{(k)\top} \bm r_{\jL{\LLL}}^{(k)}}{\bm p_{\jL{\LLL}}^{(k)\top} \bm A_{\jL{\LLL}} \bm p_{\jL{\LLL}}^{(k)}}\bm p_{\jL{\LLL}}^{(k)}.
\end{equation*}
Recall the following well-known error estimate for CG:
\begin{align}\label{iterative_convergences}
\left\|\bm c_{\jL{\LLL}}-\bm c_{\jL{\LLL}}^{(k)}\right\|_{\bm A_{\jL{\LLL}}} & \le 2 \left( \frac{\sqrt{\kappa_{\jL{\LLL}}} -1}{\sqrt{\kappa_{\jL{\LLL}}}+1} \right)^k\left\|\bm c_{\jL{\LLL}}-\bm c_{\jL{\LLL}}^{(0)}\right\|_{\bm A_{\jL{\LLL}}},
\end{align}
where $\kappa_{\jL{\LLL}} = \kappa( \bm y_{\jL{\LLL}} )$ denotes the condition number of $\bm A_{\jL{\LLL}}$, $\bm c^{(0)}_{\jL{\LLL}}$ is the vector of initial guess and $\bm c^{(k)}_{\jL{\LLL}}$ is the output of the $k$-th iteration of the CG solver. As opposed to Example \ref{ex4}, for this example Algorithm 1 accelerates the solution of the linear system \eqref{jlinsys}.
\end{example}


To evaluate the efficiency of the accelerated SC method, we define cost metrics for the construction of standard and accelerated SC approximations. 
In general, the computational cost in floating point operations (flops) is the total number iterations to solve \eqref{samplelinsys} summed over each of the sample points---denoted by $K_\mathrm{zero}$ and $K_\mathrm{acc}$ for the standard and accelerated SC methods, respectively---multiplied by the cost of performing one iteration, denoted $\mathcal{C}_{\text{iter}}$. 
Let $\mathcal{C}_\mathrm{int}$ be the additional cost of interpolation incurred by using the accelerated initial vectors \eqref{c0jcwyj}. Then, we define 
\begin{equation}\label{eq:cost0}
\mathcal{C}_\mathrm{zero} = \mathcal{C}_\mathrm{iter}\,  K_\mathrm{zero},
\end{equation}
for the standard SC approach, and
\begin{equation}\label{eq:costa}
\mathcal{C}_\mathrm{acc} = \mathcal{C}_\mathrm{iter}\, K_\mathrm{acc} + \mathcal{C}_\mathrm{int},
\end{equation}
for the accelerated SC approximation, 
respectively. 

In Example \ref{ex5} the discretization of the linear PDE leads to $M_L$ sparse systems of equations of size $M_h \times M_h$. When solving these systems with a CG solver, $K_\mathrm{zero}$ and $K_\mathrm{acc}$ are the sum of solver iterations contributed from each sample system. In this case, the cost of one iteration is just the cost of one matrix vector product, i.e., $\mathcal{C}_\mathrm{iter} = C_D M_h$, where $C_D$ depends on the domain $D$ and the type of finite element basis.
\begin{remark}\label{rem:ML} {\em (Relationship to multilevel methods).}
Multilevel methods reduce the complexity of stochastic sampling methods by balancing errors and computational cost across a sequence of stochastic and spatial approximations. Let $u_{h_k} \in V_k$, $k=0,\ldots, K$, be a sequence of semi-discrete approximations built in nested spaces, i.e., $V_0\subset\ldots\subset V_K$. Multilevel methods are based on the following identity:
	\begin{equation*}
		u_{h_K} = \sum_{k=0}^{K} ( u_{h_k} - u_{h_{k-1}} ).
	\end{equation*}
Letting $\mathcal{Q}_{\LLL_{K-k}}, k=0,\ldots,K,$ denote the chosen method of stochastic approximation, a general multilevel method might be written as
	\begin{equation*}
		u_K^\mathrm{(ML)} = \sum_{k=0}^{K} \mathcal{Q}_{\LLL_{K-k}}[ u_{h_k} - u_{h_{k-1}} ].
	\end{equation*}
The main idea is that highly resolved, expensive stochastic approximations,~e.g.,~$\mathcal{Q}_{\LLL_K}$,~in combination with coarse deterministic approximations, that is, $u_{h_0}$, and vice versa. In a similar way, collocation with nested grid points provides a natural multilevel hierarchy which we use in our method to accelerate each PDE solve \eqref{c0jcwyj}. A combination of these methods could involve using our algorithm to accelerate the construction of the operators $\mathcal{Q}_{\LLL_{K-k}}$, as well as reusing information from level to level, thus improving further the performance of SC methods.
\end{remark}
\begin{remark}  {\em (Interpolation costs).}
Note that many adaptive interpolation schemes already require evaluation of the intermediate interpolation operators as in \eqref{c0jcwyj}, e.g., to compute residual error estimators. Thus, these methods will incur the interpolation cost $\mathcal{C}_\mathrm{int}$ even in the case of zero initial vectors. Furthermore, for most nonlinear problems the deterministic solver is expensive, thus reducing the number of iterations is the most important element in reducing the cost. In each of these settings, we can define the cost metrics to simply be $K_\mathrm{zero}$ and $K_\mathrm{acc}$.
\end{remark}

\begin{remark}\label{rem:preconditioner} {\em (Hierarchical preconditioner construction).}
When solving linear systems using iterative methods, convergence properties can be improved by considering the condition number of the system. 
 As with initial vectors, an interpolation algorithm can be used to construct good, cheap preconditioners. We consider preconditioner algorithms where an explicit preconditioner matrix, or its inverse, is constructed. In this case, for some low collocation level $\LLL_\mathrm{PC}$, we construct a strong preconditioner, $P_{\jL{\LLL_\mathrm{PC}}}:= P(\bm y_{\jL{\LLL_\mathrm{PC}}})$, for each individual iterative solver, $\jcount=1,\ldots, M_{\LLL_\mathrm{PC}}$. Then, these lower level preconditioners are interpolated for the subsequent levels. More specifically, for $\LLL>\LLL_\mathrm{PC}$, and $\bm y_{\jL{\LLL}}\in\mathcal{H}_\LLL\setminus\mathcal{H}_{\LLL_\mathrm{PC}}$, we use the preconditioner
\begin{equation}\label{PCalg}
\widetilde{P}_{\jL{\LLL}} := \widetilde{P}(\bm y_{\jL{\LLL}}) =  \sum_{\jcount'=1}^{M_{\LLL_\mathrm{PC}}} P_{\jL{\LLL_\mathrm{PC}}'} ~\Psi_{\jL{\LLL_\mathrm{PC}}'}(\bm y_{\jL{\LLL}}).
\end{equation}
Numerical illustrations of this approach are given in \S\ref{sec:numerical}.
\end{remark}

\section{Applications to sparse grid stochastic collocation}\label{sec:scm}
In this section, we provide a specific example of an interpolation scheme satisfying the assumptions described in \S\ref{sec:algo}, i.e., a generalized sparse grid SC approach for a fixed level $\LLL$.
In what follows, we briefly review the construction of sparse grid interpolants, and 
rigorously analyze the approximation errors and the complexities of both the standard and accelerated SC approaches, in order to demonstrate the improved efficiency of the proposed acceleration technique when applied to iterative linear solvers. 
%

The fully discrete SC approximation is built by polynomial interpolation of the semi-discrete solution $u_h(x, \bm y)$ on an appropriate set of collocation points in $\Gamma$. 
In our setting, such an interpolation scheme is based on a sparse tensor products of one-dimensional Lagrange interpolating polynomials with global support. Specifically, in the one-dimensional case, $N=1$, we introduce a sequence of Lagrange interpolation operators $\mathscr{U}^{m(l)}: C^0(\Gamma) \rightarrow \mathcal{P}_{m(l)-1}(\Gamma)$, with $\mathcal{P}_{m(l)-1}(\Gamma)$ the space of degree $m(l)-1$ polynomials over $\Gamma$. Given a general function $v\in C^0(\Gamma)$, these operators are defined by
\begin{equation*}
\mathscr{U}^{m(l)}[v](y) = \sum_{j=1}^{\mmm(l)} v(y_{j}^{l})\; \psi_j^{l}(y).
\end{equation*}
Here $l \in \mathbb{N}$ represents the resolution level of the operator, $m(l) \in \mathbb{N}_{+}$ denotes the number of interpolation points on level $l$, $\psi^{1}_1(y) = 1$ and for $l > 1$, 
\begin{equation*} 
 	\psi_j^{l}(y) = \prod_{\substack{i=1 \\ i\neq j}}^{\mmm(l)} \frac{y - y^{l}_{i}}{y^{l}_{j}-y^{l}_{i}} \;\; \mbox{ for }\;\; j=1,\ldots,\mmm(l),
\end{equation*}
are the global Lagrange polynomials of degree $\mmm(l)-1$ associated with the point set $\vartheta^l = \{y_1^l, \ldots, y_{m(l)}^l\}$. 
To satisfy Assumption \ref{assumption2}, we need nestedness of the one-dimensional sets, i.e.,~$\vartheta^{l-1} \subset \vartheta^l$, which is determined by the choice of interpolation points and the definition of $m(l)$. 
 In addition, we remark that similar constructions for $\mathscr{U}^{m(l)}$ can be built based on wavelets \cite{gunzburger2012adaptive} or other locally supported polynomial functions \cite{gunzburger2014stochastic}.

In the multi-dimensional case, i.e., $N>1$, using the convention that $\mathscr{U}^{m(0)} = 0$, we introduce the difference operator
\begin{equation}\label{delta}
	\Delta^{m(l_1)} \otimes \cdots \otimes \Delta^{m(l_N)} = \bigotimes_{n=1}^N \left(\mathscr{U}^{m(l_n)} - \mathscr{U}^{m(l_n-1)}\right),
\end{equation}
and define the multi-index $\mathbf{l}= (l_1, \ldots, l_N)\in \mathbb{N}_{+}^N$. The desired approximation is defined by a linear combination of tensor-product operators \eqref{delta} over a set of multi-indices, determined by the condition $g(\mathbf{l}) \leq L$, for $L\in \mathbb{N}_+$, and $g(\mathbf{l}): \mathbb{N}^N_{+} \rightarrow \mathbb{N}_{+}$ a strictly increasing function. For $v \in C^0(\Gamma)$ , we now define the generalized SC operator $\mathcal{I}^{m,g}_\LLL$ by
\begin{equation} \label{eq:SG}	
\begin{aligned}
	\mathcal{I}^{m,g}_\LLL[v](\bm y) &= \sum_{g(\mathbf{l}) \leq \LLL } \left(\Delta^{m(l_1)} \otimes \cdots \otimes \Delta^{m(l_N)}\right)[v](\bm y)\\
	&= \sum_{g(\mathbf{l}) \leq \LLL }  \sum_{\mathbf{i} \in \{ 0,1\}^N} (-1)^{|\mathbf{i}|} 
	\left(\mathscr{U}^{m(l_1-i_1)} \otimes \cdots \otimes \mathscr{U}^{m(l_N-i_N)}\right) [v](\bm y),\\
\end{aligned}
\end{equation}
where $\mathbf{i}=(i_1,\ldots,i_N)$ is a multi-index with $i_n \in \{0,1\}$, $|\mathbf{i}| = i_1+\cdots + i_N$, and $\LLL \in \mathbb{N}_{+}$ represents the approximation level. This approximation lives in the tensor product polynomial space given by
\begin{equation*}
	\mathcal{P}_{\Lambda^{m,g}_\LLL} = \textrm{span} \left\{ \prod_{n=1}^N y_n^{l_n}~\bigg|~\mathbf{l}\in\Lambda^{m,g}_L  \right\},
\end{equation*}
where the multi-index set is defined as follows
\begin{equation*}
	\Lambda^{m,g}_\LLL = \left\{ \mathbf{l} \in \mathbb{N}^N~\bigg|~g( \mathbf{m}^{\dagger}( \mathbf{l} + \bm 1 )) \leq L \right\}.
\end{equation*}
Here $\mathbf{m}(\mathbf{l}) = (m(l_1),\ldots, m(l_N))$, and $m^{\dagger}(l) := \min\{w\in\mathbb{N}_+: m(w) \geq l\}$ is the left inverse of $m$ (see \cite{Back2011}).

Specific choices for the one-dimensional growth rate $m(l)$ and the function $g(\mathbf{l})$ are needed to define
the multi-index set $\Lambda^{m,g}_\LLL$ and the corresponding polynomial space $\mathcal{P}_{\Lambda^{m,g}_\LLL}$ for the approximation.
In this work, we construct the interpolant in \eqref{eq:SG} using the anisotropic Smolyak construction, i.e.,
\begin{align}\label{growthSmolyak}
m(1) = 1, \, m(l) = 2^{l-1}+1 \mbox{ for } \, l > 1 \, \mbox{ and } \, g(\mathbf{l}) = \sum_{n=1}^N \frac{\alpha_n}{\alpha_{\min}}(l_n-1),
\end{align}
where $\bm \alpha = (\alpha_1, \ldots, \alpha_N) \in \mathbb{R}_{+}^N$ is a vector of weights reflecting the anisotropy of the system, i.e., the relative importance of each dimension, with $\alpha_{\text{min}} := \min_{n} \alpha_n$ (see \cite{nobile2008anisotropic} for more details). 
Our analysis does not depend strongly on this choice of $m$ and $g$, and we could use other functions, e.g., $m(l)=l$ and $g(\mathbf{l}) = \max_{n} \alpha_n l_n$ define the anisotropic tensor product approximation. 
 %

When $\Gamma$ is a bounded domain in $\mathbb{R}^N$, a common choice is the Clenshaw-Curtis abcsissas \cite{clenshaw1960method} given by the sets of extrema of Chebyshev polynomials including the end-point extrema. For a sample set of any size $\mmm(l)>1$, the abscissas in the standard domain $[-1,1]$ are given by
\begin{equation}\label{CCpoints}
\vartheta^{l} = \left\{ y^{l}_{j} \in [-1,1] \;\bigg|\; y^{l}_{j} = -\cos\left(\frac{\pi\left(j-1\right)}{\mmm(l)-1}\right) \mbox{ for } j = 1,\ldots,\mmm(l) \right\}.
\end{equation}
By taking $y^{1}_{1}=0$ and letting $\mmm(l)$ grow according to the rule in \eqref{growthSmolyak},
 one gets a sequence of nested sets $\vartheta^{l} \subset \vartheta^{l+1}$ for $l\in \mathbb{N}_{+}$. In addition, with $g(l)$ defined as in \eqref{growthSmolyak}, the resulting set of $N$-dimensional abscissas is a Clenshaw-Curtis sparse grid.
 Other nested families of sparse grids can be constructed from, e.g., the Leja points \cite{de2004leja}, Gauss-Patterson \cite{trefethen2008gauss}, etc.
\begin{remark}\label{rem:specificSC} {\em (Specific Choice of $m,g$).}
	For the remainder of the paper, we will assume that the functions $m$ and $g$ are given as in \eqref{growthSmolyak}, and use an underlying Clenshaw-Curtis sparse grid. For simplicity, we will also only consider {\em isotropic} collocation methods, i.e. $\alpha_1=\alpha_2=\ldots=\alpha_N$. We then lighten the notation by defining $\mathcal{I}_\LLL := \mathcal{I}^{m,g}_\LLL$.
\end{remark}

Construction of the approximation $\mathcal{I}_\LLL[v] := \mathcal{I}^{m,g}_\LLL[v]$ requires evaluation of $v$ on a set of collocation points $\mcH_\LLL \subset \Gamma$ with cardinality $M_L$. In our case, since the one-dimensional point sets are nested, i.e.,~$\vartheta^{l} \subset \vartheta^{l+1}$ for $l \in \mathbb{N}_{+}$, so that the multi-dimensional point set used by $\mathcal{I}_\LLL[v]$ is given by
\[
\mcH_\LLL = \bigcup_{g(\mathbf{l})=\LLL} \left( \vartheta^{l_1} \otimes \cdots \otimes \vartheta^{l_N}  \right),
\]
and the nested structure is preserved, i.e.,~$\mcH_\LLL \subset \mcH_{\LLL+1}$, to satisfy assumption \ref{assumption2}. Define the difference of the sets $\Delta\mcH_\LLL := \mcH_\LLL\setminus\mcH_{\LLL-1}$, and the number of new collocation points $\Delta M_\LLL = \#(\Delta \mcH_\LLL)$. 
With this nestedness condition, the approximation $\mathcal{I}_\LLL[v]$ is a Lagrange interpolating polynomial \cite{nobile2008sparse}, and thus \eqref{eq:SG} can be rewritten as a linear combination of Lagrange basis functions,
\begin{equation} \label{eq:lagrangeinterp}
\begin{aligned}
\mcI_\LLL[v]( \bm y) &= \sum_{\jcount=1}^{M_\LLL}v(\bm y_{\jL{\LLL}})\Psi_{\jL{\LLL}}( \bm y)  \\
& = \sum_{\jcount=1}^{M_\LLL}v(\bm y_{\jL{\LLL}})   \underbrace{\sum_{\mathbf{l} \in \mathcal{J}(\jL{\LLL})} \sum_{ \mathbf{i} \in \{ 0,1\}^N} (-1)^{|\mathbf{i}|} \prod_{n=1}^N \psi_{k_n(j)}^{l_n-i_n}(y_n)}_{\Psi_{\jL{\LLL}}(\bm y)},\\
\end{aligned}
\end{equation}
where the index set $\mathcal{J}(\jL{\LLL})$ is defined by 
\[
\mathcal{J}(\jL{\LLL}) = \left\{ \mathbf{l} \in \mathbb{N}_{+}^N \; \Bigg| \; g(\mathbf{l}) \leq \LLL  \mbox{ and } \bm y_{\jL{\LLL}} \in \bigotimes_{n=1}^N \vartheta^{l_n-i_n} \mbox{ with } \mathbf{i} \in \{0,1\}^N\right\}.
\]
For a given $\LLL$ and $\jjj$, this represents the subset of multi-indices corresponding to the tensor-product operators $\mathscr{U}^{m(l_1-i_1)} \otimes \cdots \otimes \mathscr{U}^{m(l_N-i_N)}$ in \eqref{eq:SG} with the supporting point $\bm y_{\jL{\LLL}}$. Then for each $\mathbf{l} \in \mathcal{J}(\jL{\LLL})$ and $\mathbf{i} \in \{0,1\}^N$, the function $\prod_{n=1}^N\psi_{k_n(j)}^{l_n-i_n}(y_n)$ with $k_n(j) \in \{1, \ldots, m(l_n-i_n)\}, n = 1, \ldots, N$, represents the unique Lagrange basis function for the operator $\mathscr{U}^{m(l_1-i_1)} \otimes \cdots \otimes \mathscr{U}^{m(l_N-i_N)}$ corresponding to $\bm y_{\jL{\LLL}}$.
Therefore, the functions $\{\Psi_{\jL{\LLL}}\}_{\jcount=1}^{M_\LLL}$ are given by a linear combination of tensorized Lagrange polynomials satisfying the ``delta property'', i.e.,~$\Psi_{\jL{\LLL}'}(\bm y_{\jL{\LLL}}) = \delta_{jj'}$ for $\jcount,\jcount' = 1, \ldots, M_\LLL$, and is in the required form of \eqref{eq:general_interpolant}.

Finally, to construct the fully-discrete approximation in the space $W_h(D) \otimes \mathcal{P}_{\Lambda^{m,g}_\LLL}(\Gamma)$ we apply the interpolation operator $\mathcal{I}_\LLL[\cdot]$, given by \eqref{eq:lagrangeinterp}, to the semi-discrete solution $u_h(x,\bm y)$ in \eqref{eq:FEM} to obtain:,
\begin{equation}\label{eq:fullydiscrete}
\begin{aligned}
	u_{h,\LLL}(x, \bm y) & = \mathcal{I}_\LLL[u_h](x, \bm y) = \sum_{\jcount=1}^{M_\LLL} \left( \sum_{\iii=1}^{M_h} c_{\jL{\LLL}, \iii} \varphi_\iii(x) \right) \Psi_{\jL{\LLL}}(\bm y).
\end{aligned}
\end{equation}
Due to the delta property of the basis function $\Psi_{\jL{\LLL}}(\bm y)$, the interpolation matrix for $\mathcal{I}_\LLL[u_h]$ is a diagonal matrix, and thus the coefficient vectors $\bm c_{\jL{\LLL}} = (c_{\jL{\LLL},1}, \ldots, c_{\jL{\LLL},M_h})$ for $\jcount=1, \ldots, M_\LLL$ can be computed by {\em independently} solving $M_\LLL$ systems of type \eqref{samplelinsys}.

\subsection{Error estimates for fixed $\LLL$}
\label{sec:estimate}
In what follows, we focus on the linear elliptic problem \eqref{eq:ellip} described in Examples \ref{ex:poisson} and \ref{ex5}, and present a detailed convergence and complexity analysis of a fully discrete SC approximation, denoted $\widetilde{u}_{h,\LLL}$, for any fixed level, $1\leq L\leq\Lmax$. As specified in Remark \ref{rem:specificSC}, in this section we consider only the isotropic Smolyak version of SC interpolant given by \eqref{eq:SG}, defined on Clenshaw-Curtis abscissas, for solving the  parameterized linear elliptic PDE. However, our analysis can be extended without any essential difficulty to anisotropic SC methods and more complicated underlying PDEs.

The parameterized elliptic PDE \eqref{eq:ellip} admits a weak form
that is a symmetric, uniformly coercive and continuous bilinear operator on $H_0^1(D)$; i.e., there exist $\alpha, \beta>0$, depending on $a_\mathrm{min}$ and $a_\mathrm{max}$ but independent of $\bm y$, such that for every $v,w\in H_0^1(D)$, 
\begin{align*}
\bigg|\int_D a(\bm y) \nabla v \, \nabla w \, dx \,\bigg|  & \le \alpha \norm{v}_{H_0^1(D)} \norm{w}_{H_0^1(D)}\; \text{ and }\; \beta \norm{v}^2_{H_0^1(D)}  \le \int_D a(\bm y) | \nabla v |^2 \, dx. 
\end{align*}
In this case, the bilinear form induces a norm $\norm{v}^2 = \int_D a(\bm y) |\nabla v|^2 \, dx$, which for functions $v(x) = \sum_{i=1}^{M_h}c_i \phi_i(x) \in W_h(D)$, with $\bm c = (c_1, \ldots, c_{M_h})$, coincides with the discrete norm $\|\bm c\|_{\bm A(\bm y)}$, where the matrix $A(\bm y)$ is defined in \eqref{BJ}. Thus we have
\begin{subequations}\label{PDEassumptions}
\begin{align}
& \text{Continuity:} \;\;\; \norm{\bm c}_{\bm A(\bm y)} = \norm{v} \le \sqrt{\alpha} \norm{v}_{H_0^1(D)},\text{~and,}\label{PDEContinuity}\\ 
& \text{Ellipticity:} \;\;\;\; \sqrt{\beta} \norm{v}_{H_0^1(D)} \le \norm{v} = \norm{\bm c}_{\bm A(\bm y)}.\label{PDEEllipticity}
\end{align}
\end{subequations}
We next state some regularity conditions on the parameterized solution $u: \Gamma \rightarrow H_0^1(D)$ to the parameterized elliptic PDE in Examples \ref{ex:poisson} and \ref{ex5}.
\begin{assumption}\label{assum:reg} {\em (Polyellipse analyticity).}
Let $\bm \gamma = (\gamma_1,\ldots,\gamma_N)\in (1,\infty)^N$, and assume that $u:\Gamma \rightarrow H_0^1(D)$ admits a complex extension $u^*:\mathbb{C}^N \rightarrow H_0^1(D)$, which is analytic on the polyellipse 
\begin{equation*}
	\Sigma(\bm \gamma) = \prod_{1\leq n\leq N} \Sigma(n;\gamma_n)\subset\mathbb{C}^N,
\end{equation*}
where $\Sigma(n;\gamma_n)$ denotes the region bounded by the Bernstein ellipse,
\begin{equation*}
	\Sigma(n;\gamma_n) = \left\{ \frac12 \left(z_n + z_n^{-1}\right) : z_n\in\mathbb{C}, |z_n |\leq \gamma_n \right\}.
\end{equation*} 
\end{assumption}

The set $\Sigma(\bm\gamma)\subset\mathbb{C}^N$ is the product of ellipses in the complex plane, with foci $z_n = \pm 1$, which are the endpoints of the domain $\Gamma_n, n=1,\ldots,N$. Such ellipses are common in proving convergence results for global interpolation schemes. Conditions under which $u$ satisfies Assumption \ref{assum:reg} can be found in \cite[Theorem 1.2]{cohen2011analytic} and \cite[Theorem 2.5]{devore2015quasi}. 


In order to investigate the complexity of the fully discrete approximation $\widetilde{u}_{h.\LLL}, L\in\mathbb{N}_+$, we first need to derive sufficient conditions for the error  $\|u - \widetilde{u}_{h,\LLL}\|_{L^2_\varrho}$ to achieve a tolerance of $\varepsilon>0$, where
 $L^2_\varrho := L^2_\varrho(\Gamma;H_0^1(D))$. Using the triangle inequality, the total error can be split into three parts, i.e.,
\begin{equation} \label{3termError} 
\norm{u - \widetilde{u}_{h,\LLL}}_{L^2_\varrho} \le \underbrace{\norm{u - u_{h}}}_{e_1} \hspace{-0.05cm}{}_{L^2_\varrho}+ \underbrace{\norm{u_{h} - u_{h,\LLL}} }_{e_2}\hspace{-0.05cm}{}_{L^2_\varrho} + \underbrace{\norm{u_{h,\LLL} - \widetilde{u}_{h,\LLL}}}_{e_3}\hspace{-0.05cm}{}_{L^2_\varrho}.
\end{equation}
The contributions of $e_1$ and $e_2$ correspond to the FEM and SC errors, respectively, and have been previously examined \cite{nobile2008sparse}. 
The error $e_3$ contributed by the linear solver is often omitted from the analysis in the literature, and in practice can be controlled by setting a tight tolerance on the iterative solver. However, the analysis presented here is focused on providing cost estimates for the iterative solver and requires careful consideration of this term. 
First, we recall error estimates for $e_1$ and $e_2$, given from 
\cite{nobile2008sparse}. 
{\begin{lemma}\label{assum:FEMconv}
Let  $\mathcal{T}_h$ be a uniform finite element mesh over $D\subset\mathbb{R}^d, d=1,2,3,$ with $M_h = \mathcal{O}(1/h^d)$ grid points. For the random elliptic PDE in Example \eqref{ex:poisson}, when $u(x,\bm y) \in L_{\varrho}^2(\Gamma; H_0^{1}(D)\cap H^{s+1}(D)), s\in\mathbb{N}_+$, the error of the finite element approximation $u_h$ is bounded by
\begin{align}\label{FEMerror}
\norm{u-u_h}_{L^2_\varrho} \le C_{\text{\em fem}} \, h^{s}, 
\end{align}
where the constant $C_{\text{\em fem}}$ is independent of $h$ and $\bm y$. 
\end{lemma}
}

\begin{lemma}\label{thm:nobile}
Let $u$ satisfy Assumption \ref{assum:reg}. For $\LLL\in\mathbb{N}^{+}$, the interpolation error $u - \mathcal{I}_\LLL[u]$ of the sparse grid SC method using Clenshaw-Curtis abscissas can be bounded as
\begin{align}\label{AsymptinterperrorSG}
    \norm{u - \mcI_\LLL[u]}_{L^\infty(\Gamma;H_0^1(D))} 
    	& \le C_{\text{\em sc}} \mathrm{e}^{-rN2^{\LLL/N}},
\end{align}
where, for a constant $0<\delta<1$, the rate $r = (1-\delta)\min_{1\le n\le N} \log\gamma_n$,
and the constant $C_{\text{\em sc}}>0$ depends on $N$, $u$, and $\delta$. 
\end{lemma}
We remark that the projection of $u$ into the finite element subspace, denoted $u_h$, also satisfies Assumption \ref{assum:reg} with the same region of analyticity, and therefore the application of the interpolant, $\mathcal{I}_\LLL$, to the semidiscete solution $u_h$ will converge as in \eqref{AsymptinterperrorSG}. 

We now turn our attention to the global solver error $e_3$ in \eqref{3termError}, which is the error incurred from approximating the solution to \eqref{jlinsys} at each sample point. 
The difference $u_{h,\LLL} - \widetilde{u}_{h,\LLL}$ can be written as an interpolant of the solver error, i.e.,
\begin{equation*} 
u_{h,\LLL} - \widetilde{u}_{h,\LLL} = \mcI_\LLL [u_h - \widetilde{u}_h],
\end{equation*}
which represents the solver error amplified by the interpolation operator. 
 For the operator $\mathcal{I}_\LLL[\cdot]$ in \eqref{eq:lagrangeinterp}, we have
\begin{equation*}
\norm{u_{h,\LLL} - \widetilde{u}_{h,\LLL}}_{L^\infty(\Gamma;H_0^1(D))} \le \LC{\LLL} \max_{\jcount=1,\ldots,M_\LLL}\norm{u_h(\bm y_{\jL{\LLL}}) - \widetilde{u}_h(\bm y_{\jL{\LLL}})}_{H_0^1(D)}.
\end{equation*}
Thus, from the ellipticity condition in (\ref{PDEEllipticity}),
\begin{equation*}
\begin{aligned}
	e_3 & \leq \LC{\LLL} \max_{\jcount=1,\ldots,M_\LLL}\norm{u_h(\bm y_{\jL{\LLL}}) - \widetilde{u}_h(\bm y_{\jL{\LLL}})}_{H_0^1(D)} 
		\leq  \LC{\LLL} \frac{1}{\sqrt{\beta}}\max_{\jcount=1,\ldots,M_\LLL} \norm{\bm c_{\jL{\LLL}} - \widetilde{\bm c}_{\jL{\LLL}}}_{\bm A(\bm y_{\jL{\LLL}})}
		 \le \frac{\tau}{\sqrt{\beta}}\,  \LC{\LLL},
\end{aligned}
\end{equation*}
where $\tau$ is defined to be the tolerance of the linear solver. Note that the expression $u_h - \widetilde{u}_h$ is only defined at collocation points. 
 The solver error for each fixed $\bm y_{\jL{\LLL}}\in \mathcal{H}_\LLL$ is controlled by the CG convergence estimate (\ref{iterative_convergences}). 
 The Lebesgue constant of the operator $\mathcal{I}_\LLL[\cdot]$ is defined by $\LC{\LLL} = \max_{\bm y\in\Gamma}\sum_{\jcount=1}^{M_\LLL} | \Psi_{\jL{\LLL}}(\bm y)|$ where $\Psi_{\jL{\LLL}}$ is given in \eqref{eq:lagrangeinterp}. 
We now provide an upper bound of $\LC{\LLL}$ in the following lemma.

\begin{lemma}\label{lem:LebesgueSG}
The Lebesgue constant for the isotropic sparse-grid interpolation operator $\mcI_\LLL[\cdot]$ in \eqref{eq:lagrangeinterp} using the Clenshaw-Curtis rule on $\Gamma = \prod_{n=1}^N \Gamma_n = [-1,1]^N$ is bounded by
\begin{equation}
\LC{\LLL} \le \left[(\LLL+1)(\LLL+2)\right]^N,
\end{equation}
where $\LLL$ and $N$ are the level of the interpolation operator and dimension of the parameter space, respectively.
\end{lemma}
\begin{proof}
For each $n = 1, \ldots, N$, recall that the Lebesgue constants $\lambda_{l_n}$ of the one-dimensional operators $\mathscr{U}^{m(l_n)}$ are given by 
\cite{trefethen2013approximation}
\[
\lambda_{l_n} = \max_{z \in \Gamma_n} \sum_{\jjj=1}^{m(l_n)} \left| \psi^{l_n}_\jjj (z) \right| .
\]
For Lagrange interpolants based on Clenshaw-Curtis abscissas \eqref{CCpoints}, we have \cite{dzjadyk1983asymptotics}
\[
\lambda_{l_n} \le \frac{2}{\pi}\log\left(m\left(l_n\right)-1\right) + 1\; \text{ for } l_n \ge 2.
\]
Combining this with the growth rate 
 $m(l_n)=2^{l_n-1}+1$ for $l_n \ge 2$ 
 given by \eqref{growthSmolyak}, it is easy to obtain that
\begin{equation*}
\lambda_{l_n} \le 2l_n-1 \; \text{ for }\; l_n \ge 2.
\end{equation*} 
For $v\in C^0(\Gamma_n)$, the difference operator $\Delta^{m(l_n)}$ for $l_n=1$ satisfies
\begin{equation*}
\norm{\Delta^{m(1)} [v] }_{L^\infty(\Gamma_n)} = \norm{\mathscr{U}^{m(1)} [v]}_{L^\infty(\Gamma_n)}\le \lambda_{1} \max_{y_n\in \vartheta^{1}} |v(y_n)|.
\end{equation*}
For $l_n\ge 2$, the triangle inequality yields
\begin{align*}
\norm{\Delta^{m(l_n)} [v]}_{L^\infty(\Gamma_n)}  &= \norm{\mathscr{U}^{m(l_n)} [v] -\mathscr{U}^{m(l_n-1)} [v] }_{L^\infty(\Gamma_n)} \\
& \le (\lambda_{l_n}+ \lambda_{l_n-1})\max_{y_n\in\vartheta^{l_n}} |v(y_n)|.
\end{align*}
Finally, for $v \in C^0(\Gamma)$, we bound the interpolant $\mathcal{I}_\LLL[v]$ by
\begin{align*}
\norm{\mcI_\LLL[ v ]}_{L^\infty(\Gamma)} & = \left\| \sum_{g(\mathbf{l}) \leq \LLL} \Delta^{m(l_1)} \otimes \cdots \otimes \Delta^{m(l_N)}[v] \right\|_{L^{\infty}(\Gamma)} \notag\\
& \le  \left(2^N \sum_{g(\mathbf{l}) \leq \LLL} \prod_{n=1}^N l_n \right) \max_{\jcount=1,\ldots,M_\LLL} |v(\bm y_{\jL{\LLL}})|  \le 2^N \left(\sum_{l=1}^{\LLL+1}l \right)^N \max_{\jcount=1,\ldots,M_\LLL} |v(\bm y_{\jL{\LLL}})|  \notag\\
& = \left[(\LLL+1)(\LLL+2)\right]^N \max_{\jcount=1,\ldots,M_\LLL} |v(\bm y_{\jL{\LLL}})|,
\end{align*}
which gives the desired estimate.
\end{proof}

\subsection{Complexity analysis}
\label{sec:complexity}
Now we analyze the cost of constructing $\widetilde{u}_{h,\Lmax}, \Lmax\in\mathbb{N}_+$, with the prescribed accuracy $\varepsilon$.  Here we assume $\varepsilon>0$ is sufficiently small, and study the asymptotic growth of the total costs \eqref{eq:costa} for the accelerated construction of $\widetilde{u}_{h,\Lmax}$, described in \S\ref{sec:algo}. For comparison, we will also analyze the cost \eqref{eq:cost0} associated with the standard SC method, where iterative solvers for the sequence of solutions to the linear systems \eqref{jlinsys} are seeded with the zero vector as an initial guess.
According to the error estimates discussed in \S \ref{sec:estimate}, a sufficient condition to ensure $\|u - \widetilde{u}_{h,\Lmax} \|_{L_\varrho^2} \le \varepsilon$ is that
\begin{subequations}\label{e1e2e3}
\begin{align}
& \|e_1\|_{L_\varrho^2} \le C_{\text{fem}} h^s \le \frac{\varepsilon}{3}, \label{e1}\\
& \|e_2\|_{L_\varrho^2} \le \|e_2\|_{L_\varrho^{\infty}} \le C_{\text{sc}}\, \mathrm{e}^{-rN2^{\Lmax/N}} \le \frac{\varepsilon}{3},\label{e2}\\
& \|e_3\|_{L_\varrho^2} \le \|e_3\|_{L_\varrho^{\infty}} \le (\Lmax+2)^{2N} \frac{\tau}{\sqrt{\beta}} \le \frac{\varepsilon}{3}. \label{e3}
\end{align}
\end{subequations}
In section \S\ref{sec:algo} we defined $K_\mathrm{zero}$ and $K_\mathrm{acc}$ as the total number of solver iterations used by the standard and accelerated SC methods, respectively, to solve \eqref{jlinsys} at each sample point. Now let $K_{\text{zero}}(\varepsilon)$ and $K_{\text{acc}}(\varepsilon)$ represent the minimum values of $K_{\text{zero}}$ and $K_{\text{acc}}$, respectively, needed to satisfy the inequalities \eqref{e1e2e3}. Here we aim to estimate upper bounds of $K_{\text{zero}}(\varepsilon)$ and $K_{\text{acc}}(\varepsilon)$. Note that, for fixed dimension $N$, level $\Lmax$, and mesh size $h$, the total number of iterations is determined by the inequality $\eqref{e3}$. Larger values of $\Lmax$ and $1/h$, lead to higher costs. Thus, the estimation of $K_{\text{zero}}(\varepsilon)$ and $K_{\text{acc}}(\varepsilon)$ has two steps: (i) Given $N$ and $\varepsilon$, estimate the maximum possible $h$ to satisfy \eqref{e1} and the minimum $\Lmax$ that achieves \eqref{e2}; (ii) Substitute the obtained values into \eqref{e3} to estimate upper bounds on $K_{\text{zero}}(\varepsilon)$ and $K_{\text{acc}}(\varepsilon)$ according to the CG error estimate \eqref{iterative_convergences}. 
For (i), we have the following lemma, that follows immediately from Lemmas \ref{assum:FEMconv} and \ref{thm:nobile}.


\begin{lemma}\label{lem:hLestimates}
Given the assumptions of Lemmas \ref{assum:FEMconv} and \ref{thm:nobile}, 
the error bounds \eqref{e1} and \eqref{e2} can be achieved by choosing finite element mesh size $h$ and the sparse-grid level $\Lmax$ according to
\begin{equation}\label{lem:hofepsilon}
h(\varepsilon) = \left(\frac{\varepsilon}{3C_{\text{\em fem}}} \right)^{1/s} \;\; \text{ and }\;\;
\Lmax(\varepsilon) =  \left\lceil\frac{N}{\log 2}\log\left(\frac{1}{rN}\log\left(\frac{3 C_{\text{\em sc}}}{\varepsilon}\right)\right)\right\rceil.
\end{equation}
\end{lemma}
For convenience, we treat the integer quantities $K_{\text{zero}}(\varepsilon)$, $K_{\text{acc}}(\varepsilon)$, and $\Lmax(\varepsilon)$ as positive real numbers in the rest of this section. Now, based on the estimate in Lemma \ref{lem:LebesgueSG} for the Lebesgue constant $\LC{\Lmax}$, we state the following lemma related to the choice of an appropriate tolerance $\tau(\varepsilon)$ to satisfy the error bounds \eqref{e3}.
\begin{lemma}\label{lem_tau}
Let $\varepsilon >0$. Given the assumptions of Lemmas \ref{assum:FEMconv} and \ref{thm:nobile}, a sufficient condition to ensure $e_3 < \varepsilon/3$ is that
\begin{equation}\label{tau1}
\tau(\varepsilon) = \frac{\sqrt{\beta}\, \varepsilon}{3(\Lmax(\varepsilon)+2)^{2N}}.
\end{equation}
Moreover, it holds
\begin{equation*}
\frac{1}{\sqrt{\beta}}(\LLL+2)^{2N} \tau(\varepsilon) \le C_{\text{\em sc}}\; \mathrm{e}^{-rN2^{\LLL/N}}\;\; \text{for}\;\;
\LLL = 0, \ldots, \Lmax(\varepsilon)-1,
\end{equation*}
where $\Lmax(\varepsilon)$ is the minimum level given in \eqref{lem:hofepsilon}.
\end{lemma}

\begin{proof}
It is easy to see that \eqref{tau1} is an immediate result of \eqref{e3}. For $\LLL = 0, \ldots, \Lmax(\varepsilon)-1$, we have
\begin{equation*}
\frac{1}{\sqrt{\beta}}(\LLL+2)^{2N} \tau(\varepsilon) \le \frac{1}{\sqrt{\beta}}(\Lmax(\varepsilon)+2)^{2N} \tau(\varepsilon) \le  \frac{\varepsilon}{3} \le C_{\text{sg}}\; \mathrm{e}^{-rN2^{(\Lmax(\varepsilon)-1)/N}} \le C_{\text{sg}}\; \mathrm{e}^{-rN2^{\LLL/N}},
\end{equation*}
which completes the proof.
\end{proof}

Using the selected $h:=h(\varepsilon)$, $\Lmax:=\Lmax(\varepsilon)$, and $\tau:=\tau(\varepsilon)$, we now estimate the
upper bounds on the number of CG iterations needed to solve a linear system at a point $\bm y_{\jL{\Lmax}} \in \mcH_{\Lmax}$. 
To proceed, define
\[
k_{\text{zero}} := \max_{\bm y_{\jL{\Lmax}} \in \mathcal{H}_{\Lmax}} k_{\jL{\Lmax}} \quad \text{and} \quad k_{\text{acc}}^\LLL := \max_{\bm y_{\jL{\LLL}} \in \Delta \mathcal{H}_{\LLL}} k_{\jL{\LLL}} \;\text{ for } \; \LLL = 1, \ldots, \Lmax,
\]
where $k_{\jL{\LLL}}$ is the 
number of CG iterations required to achieve $\|\bm c_{\jL{\LLL}} - \bm c_{\jL{\LLL}}^{(k_{\jL{\LLL}})} \|_{\bm A_{\jL{\LLL}}} \le \tau(\varepsilon)$, which, in general, depends on the choice of initial vector. 
Note that, in the case $\bm c_{\jL{\LLL}}^{(0)} = (0,\ldots, 0)^{\top}$, there is no improvement in the iteration count as the level $\LLL$ increases, so $k_\text{zero}$ does not depend on $\LLL$.
Now we give the following estimates on $k_{\text{zero}}$ and $\{k_{\text{acc}}^\LLL\}_{\LLL=1}^{\Lmax}$.

{\begin{lemma} \label{lemmaSG}
Under the conditions of Lemmas \ref{assum:FEMconv} and \ref{thm:nobile}, for any $\bm y_{\jL{\Lmax}} \in \mcH_{\Lmax}$, if the CG method with zero initial vector is used to solve \eqref{jlinsys} to tolerance $\tau>0$, then $k_{\rm{zero}}$ can be bounded by
%
\begin{equation}\label{kj1}
k_\mathrm{zero} \le \log\left(\frac{2 \sqrt{\alpha} \norm{ u_h }_{L^\infty(\Gamma;H_0^1(D))} }{\tau}\right) \Bigg/ \log\left( \frac{\sqrt{\bar{\kappa}} +1}{\sqrt{\bar{\kappa}}-1} \right).
\end{equation}
Here $\overline{\kappa} = \sup_{\bm y \in \Gamma}\kappa(\bm y)$, with $\kappa(\bm y)$ the condition number of
the matrix $\bm A( \bm y )$ corresponding to \eqref{samplelinsys}. Alternatively, if the initial vector is given by the acceleration method as in \eqref{c0jcwyj}, 
then $k_{\text{acc}}^\LLL$ can be bounded by
\begin{equation}\label{kj2}
k_\mathrm{acc}^{\LLL} \le \log\left(\frac{4\sqrt{\alpha}C_{\text{\em sc}}\;\mathrm{e}^{-rN2^{(\LLL-1)/N}}}{\tau}\right) \Bigg/ \log\left( \frac{\sqrt{\bar{\kappa}} + 1}{\sqrt{\bar{\kappa}}-1} \right),
\end{equation}
for $\LLL = 1, \ldots, \Lmax$.
\end{lemma}
}

\begin{proof}
%
%
Let $\bm y_{\jL{\LLL}}$  be an arbitrary point in $\mathcal{H}_{\LLL}, 1\leq \LLL\leq \Lmax$. Given an initial guess $\bm c_{\jL{\LLL}}^{(0)}$, the minimum number of CG iterations needed to achieve tolerance $\tau>0$ can be obtained immediately from \eqref{iterative_convergences}, that is,
\begin{align*}
	k_{\jL{\LLL}} 
 = \Bigg\lceil \log \left( \frac{2\|\bm c_{\jL{\LLL}} - \bm c_{\jL{\LLL}}^{(0)}\|_{\bm A_{\jL{\LLL}}}} {\tau} \right) \Bigg/ \log\left( \frac{\sqrt{\kappa_{\jL{\LLL}}} + 1}{\sqrt{\kappa_{\jL{\LLL}}} - 1} \right) \Bigg\rceil,
\end{align*}
where $\bm A_{\jL{\LLL}}=\bm A(\bm y_{\jL{\LLL}})$ is the FE system matrix corresponding to parameter $\bm y_{\jL{\LLL}}$, and $\kappa_{\jL{\LLL}} = \kappa(\bm y_{\jL{\LLL}})$ 
is the condition number of $\bm A_{\jL{\LLL}}$ (See Example \ref{ex5}). 
In the case that $\bm c_{\jL{\LLL}}^{(0)} = (0,\ldots,0)^{\top}$, the estimate in \eqref{kj1} can be obtained from \eqref{PDEContinuity}, i.e., 
\begin{align*}
\left\| \bm c_{\jL{\LLL}}-\bm c_{\jL{\LLL}}^{(0)}\right\|_{\bm A_{\jL{\LLL}}} = \|\bm c_{\jL{\LLL}}\|_{\bm A_{\jL{\LLL}}}  & \le \sqrt{\alpha} \norm{ u_h }_{L^\infty(\Gamma;H_0^1(D))}.
\end{align*}
Alternatively, when using $\widetilde{u}_{h,\LLL-1}$ for $\LLL=1,\ldots \Lmax$ to provide initial vectors for the CG solver (based on \eqref{c0jcwyj}), for $\bm y_{L,j}\in\Delta\mathcal{H}_{L}$ we use Lemma \ref{lem_tau} and \eqref{PDEContinuity} to get the following estimate:
\begin{align*}
\left\|\bm c_{\jL{\LLL}} - \bm c^{(0)}_{\jL{\LLL}}\right\|_{\bm A_{\jL{\LLL}}} 
& \le \sqrt{\alpha} \norm{u_h - \widetilde{u}_{h,\LLL-1}}_{L^\infty(\Gamma;H_0^1(D))}\\
& \le \sqrt{\alpha} \left( \norm{u_h - u_{h,\LLL-1}}_{L^\infty(\Gamma;H_0^1(D))} + \norm{u_{h,\LLL-1} - \widetilde{u}_{h,\LLL-1}}_{L^\infty(\Gamma;H_0^1(D))} \right)\\
& \le \sqrt{\alpha} \bigg( C_{\text{sc}}\, \mathrm{e}^{-rN2^{(\LLL-1)/N}} + \frac{1}{\sqrt{\beta}}(\LLL+1)^{2N}\tau \bigg)\\
& \le 2\sqrt{\alpha} C_{\text{sc}}\, \mathrm{e}^{-rN2^{(\LLL-1)/N}}.
\end{align*}
This leads directly to the estimate in \eqref{kj2}.
\end{proof}

In the accelerated case, 
 the sparse-grid interpolant $\mathcal{I}_{\Lmax}[u_h]$ must be constructed in the following fashion:~before solving the system \eqref{jlinsys} corresponding to a sample point $\bm y_{\jL{\LLL}}\in\Delta\mcH_{\LLL}$, we must first solve the systems for all sample points in $\mcH_{\LLL-1}$.
%
 With a total number $\Delta M_\LLL= \#(\Delta\mcH_\LLL)$ of new linear systems at level $\LLL$, the total number of CG iterations for the newly added points at level $\LLL$ can be bounded by $\Delta M_\LLL k_\mathrm{zero}$ and $\Delta M_\LLL k_\mathrm{acc}^{\LLL}$,
 for the standard and the accelerated cases, respectively. 
Then since $M_{\Lmax} = \sum_{\LLL=1}^{\Lmax} \Delta M_\LLL$, we find that the total number of iterations for the standard and accelerated schemes can be bounded as
\[K_{\rm{zero}}(\varepsilon) \leq 
M_{\Lmax} \, k_{\text{zero}}, \quad \text{and} \quad K_{\rm{acc}}(\varepsilon) \leq \sum_{\LLL=1}^{\Lmax}\Delta M_\LLL \, k_{\text{acc}}^\LLL.\]
This leads to the following estimates.

\begin{theorem}\label{thm:NCGzero}
Given Assumption \ref{assum:reg}, and the conditions of Lemmas \ref{assum:FEMconv} and \ref{thm:nobile}, 
 for $\varepsilon > 0$, the
 minimum total number of CG iterations $K_\mathrm{zero}(\varepsilon)$ to achieve $\| u - \widetilde{u}_{h,\Lmax} \|_{L^2_\varrho} < \varepsilon$, using zero initial vectors is bounded by
 \begin{equation}\label{Kzero}
\begin{aligned}
K_\mathrm{zero}(\varepsilon) 
 & \le 
 C_1 \left[\log\left(\frac{3 C_{\text{\em sc}}}{\varepsilon}\right) \right]^{N} 
 \left[C_2+ \frac{1}{\log{2}} \log\log\left( \frac{3 C_{\text{\em sc}}}{\varepsilon} \right) \right]^{N-1}\\
 &\hspace{0.5cm} \times  \frac{1}{\log\left( \frac{\sqrt{\overline{\kappa}} +1}{\sqrt{\overline{\kappa}}-1} \right)}
      \left\{ \log\left(\frac{C_3}{\varepsilon}\right) + C_4 + 2N \log\log\left[\frac{1}{rN} \log\left( \frac{3C_{\text{\em sc}}}{\varepsilon}  \right)\right] \right\},
\end{aligned}
\end{equation}
where
$\overline{\kappa}$
 is as defined in Lemma \ref{lemmaSG},
and the constants $C_1$, $C_2$, $C_3$ and $C_4$ are defined by
\begin{equation}\label{CC1}
\begin{aligned}
&C_1 = \left(\frac{\mathrm{e}}{\log{2}}\right)^{N-1} \left(\frac{2}{rN}\right)^{N}, \quad
C_2 = 1 + \frac{1}{\log{2}}\log\left(\frac{1}{rN}\right),\quad \\
&C_3 =6 \sqrt{\frac{\alpha}{\beta}} \norm{ u_h }_{L^\infty(\Gamma;H_0^1(D))}, \quad
C_4= 2N \log\left( \frac{2N}{\log{2}}\right).
\end{aligned}
\end{equation}
%
\end{theorem}
\begin{proof}
To achieve the prescribed error, we balance the three error sources that contribute to the total error \eqref{3termError}. To control $e_1$ and $e_2$, set $h=h(\varepsilon)$ and $\Lmax=\Lmax(\varepsilon)$ according to Lemma \ref{lem:hLestimates}. For the solver error $e_3$, we choose the solver tolerance $\tau = \tau(\varepsilon)$ according to Lemma \ref{lem_tau}.
Then, the total number of iterations
 $K_\mathrm{zero}(\varepsilon)$
 can be bounded by
\begin{equation} \label{Kmin1}
K_\mathrm{zero}(\varepsilon)  = \sum_{\jcount=1}^{M_{\Lmax}} \le 
M_{\Lmax}\, k_\mathrm{zero}.
\end{equation}
From Lemma \ref{lem_tau} and \ref{lemmaSG}, we have
\begin{align}
k_\mathrm{zero} & \le  \log\left(\frac{2 \sqrt{\alpha} \norm{ u_h }_{L^\infty(\Gamma;H_0^1(D))} }{\tau} \right) \Bigg/ \log\left( \frac{\sqrt{\overline{\kappa}} +1}{\sqrt{\overline{\kappa}}-1} \right) \notag\\
	& \le   \log\left(\frac{6 \sqrt{\alpha} \norm{ u_h }_{L^\infty(\Gamma;H_0^1(D))} (\Lmax+2)^{2N} }{\sqrt{\beta} \varepsilon} \right) \Bigg/ \log\left( \frac{\sqrt{\overline{\kappa}} +1}{\sqrt{\overline{\kappa}}-1} \right) \label{kzero} \\
    & \le \left[ \log\left(\frac{C_3}{\varepsilon}\right) + 2N\log\left( \Lmax+2 \right)\right] \Bigg/ \log\left( \frac{\sqrt{\overline{\kappa}} +1}{\sqrt{\overline{\kappa}}-1} \right) \notag\\
    %
    & \le 
    \left\{ \log\left(\frac{C_3}{\varepsilon}\right) + C_4 + 2N \log\log\left(\frac{1}{rN} \log\left( \frac{3C_{\text{sc}}}{\varepsilon}  \right)\right) \right\} \Bigg/ \log\left( \frac{\sqrt{\overline{\kappa}} +1}{\sqrt{\overline{\kappa}}-1} \right).\notag
\end{align}
%
In addition, following \cite[Lemma 3.9]{nobile2008sparse}, we bound the number of interpolation points:
\begin{align}\label{Mestimate}
M_{\Lmax} & \le \sum_{\LLL=1}^{\Lmax} 2^\LLL \binom{N-1+\LLL}{N-1} \le \sum_{\LLL=1}^{\Lmax} 2^\LLL \left(1+\frac{\LLL}{N-1}\right)^{N-1}\mathrm{e}^{N-1}\notag\\
& \le \mathrm{e}^{N-1} 2^{\Lmax+1} \left(1+\frac{\Lmax}{N-1}\right)^{N-1}\\
& \le 2 \mathrm{e}^{N-1} \left\{\log\left(\frac{3 C_{\text{sc}}}{\varepsilon}\right) \right\}^{N} 
 \left\{C_2+ \frac{1}{\log{2}} \log\log\left( \frac{3 C_{\text{sc}}}{\varepsilon} \right) \right\}^{N-1},\notag
\end{align}
where in the last line we have used \eqref{lem:hofepsilon} to replace $\Lmax$. Substituting \eqref{kzero} and \eqref{Mestimate} into \eqref{Kmin1} concludes the proof.
\end{proof}
%
%
\begin{theorem}\label{thm:NCGaccel}
Given Assumption \ref{assum:reg}, and the conditions of Lemmas \ref{assum:FEMconv} and \ref{thm:nobile}, for $\varepsilon > 0$,
the minimum total number of CG iterations $K_\mathrm{acc}(\varepsilon)$, to achieve $\norm{u - \widetilde{u}_{h,\Lmax}}_{L^2_\varrho} < \varepsilon$, in Algorithm 1, is bounded by
\begin{equation}\label{Kacc}
\begin{aligned}
 K_\mathrm{acc}(\varepsilon)
& \le C_1 \left[\log\left(\frac{3 C_{\text{\em sc}}}{\varepsilon}\right) \right]^{N}  
 \left[C_2 + \frac{1}{\log{2}} \log\log\left( \frac{3 C_{\text{\em sc}}}{\varepsilon}\right) \right]^{N-1}\\
 & \hspace{0.5cm}\times  \dfrac{1}{\log\left( \frac{\sqrt{\overline{\kappa}} +1}{\sqrt{\overline{\kappa}}-1} \right)}
 \left\{ C_5 + 2\left(2^{\frac{1}{N}}-1\right)\log\left(\frac{3C_{\text{\em sc}}}{\varepsilon}\right) +  2N\log\log\left[\frac{1}{rN} \log\left( \frac{3C_{\text{\em sc}}}{\varepsilon}  \right)\right] \right\},
\end{aligned}
\end{equation}
where $\overline{\kappa} = \sup_{\bm y \in \Gamma}(\kappa(\bm y))$ as in Lemma \ref{lemmaSG}, $C_1$ and $C_2$ are defined as in \eqref{CC1}, and $C_5$ is defined by
\begin{equation*}
C_5 = 2N \log\left( \frac{2N}{\log{2}}\right)+ \log \left( 4\sqrt{\frac{\alpha}{\beta}} \right).
\end{equation*}
%
\end{theorem}
\begin{proof}
To achieve the prescribed error, we again choose $h = h(\varepsilon)$, $\Lmax = \Lmax(\varepsilon)$ and $\tau = \tau(\varepsilon)$ as in Lemmas \ref{lem:hLestimates} and \ref{lem_tau}. Then, the total number of iterations
 $K_\mathrm{acc}(\varepsilon)$
 can be bounded by
\begin{align*}
		K_\mathrm{acc}(\varepsilon)
        &= \sum_{\LLL=1}^{\Lmax} \sum_{\bm y_{\jL{\LLL}}\in\Delta\mcH_\LLL} k_{\jL{\LLL}} \, \le   \sum_{\LLL=1}^{\Lmax} \Delta M_\LLL \,k^{\LLL}_{\text{acc}}.
\end{align*}
From Lemma \ref{lem_tau} and \ref{lemmaSG}, for $\LLL = 1, \ldots, \Lmax$, we have
\begin{align*}
    k^{\LLL}_{\text{acc}} &\le \log\left(\frac{4 \sqrt{\alpha} C_{\text{sc}}\, \mathrm{e}^{ -rN2^{(\LLL-1)/N}}}{\tau} \right) \Bigg/ \log\left( \frac{\sqrt{\overline{\kappa}} +1}{\sqrt{\overline{\kappa}}-1} \right) \\
&\le \frac{1}{ \log\left( \frac{\sqrt{\overline{\kappa}} +1}{\sqrt{\overline{\kappa}}-1} \right)} \log\left(\frac{12 \sqrt{\alpha} C_{\text{sc}}\LC{\Lmax}\mathrm{e}^{ -rN2^{(\LLL-1)/N}}}{\sqrt{\beta}\varepsilon} \right) \\
&= \frac{1}{ \log\left( \frac{\sqrt{\overline{\kappa}} +1}{\sqrt{\overline{\kappa}}-1} \right)} \log\left[ \left( \frac{3 C_{\text{sc}} \mathrm{e}^{-rN2^{\LLL/N}}}{\varepsilon}\right) 4\sqrt{\frac{\alpha}{\beta}}\LC{\Lmax} \mathrm{e}^{ rN2^{\LLL/N}-rN2^{(\LLL-1)/N}} \right] \\
& \le  \frac{1}{ \log\left( \frac{\sqrt{\overline{\kappa}} +1}{\sqrt{\overline{\kappa}}-1} \right)} \log\left( 4\sqrt{\frac{\alpha}{\beta}}\LC{\Lmax} \mathrm{e}^{ rN\left(2^{\LLL/N} - 2^{(\LLL-1)/N}\right)}\right)  \\
& =   \frac{1}{ \log\left( \frac{\sqrt{\overline{\kappa}} +1}{\sqrt{\overline{\kappa}}-1} \right)} \left[ \log\left( 4\sqrt{\frac{\alpha}{\beta}}\LC{\Lmax}\right) + rN\left(2^{\LLL/N} -2^{(\LLL-1)/N}\right) \right] .
\end{align*}
Hence,
	\begin{align*}
		K_\mathrm{acc}(\varepsilon) \le \,& M_{\Lmax}\frac{\log\left( 4\sqrt{\alpha/\beta}\LC{\Lmax}\right)}{ \log\left( \frac{\sqrt{\overline{\kappa}} +1}{\sqrt{\overline{\kappa}}-1}\right)} 
        + \frac{rN}{ \log\left( \frac{\sqrt{\overline{\kappa}} +1}{\sqrt{\overline{\kappa}}-1} \right)}\underbrace{\sum_{\LLL=1}^{\Lmax} \Delta M_\LLL \left(2^{\Lmax/N} -2^{(\LLL-1)/N}\right)}_{S},
	\end{align*}
where $S$ can be bounded using results from geometric sums, i.e.,
\begin{align*}
S
&\le \sum_{\LLL=1}^{\Lmax}  2^\LLL \binom{N-1+\LLL}{N-1} \left(2^{\Lmax/N} -2^{(\LLL-1)/N}\right)\\
 &\le \mathrm{e}^{N-1} \left(1+ \frac{\Lmax}{N-1}\right)^{N-1}\sum_{\LLL=1}^{\Lmax} \left(2^{\Lmax/N} -2^{(\LLL-1)/N}\right) 2^\LLL \\
 & = \mathrm{e}^{N-1} \left(1+ \frac{\Lmax}{N-1}\right)^{N-1 }\left\{ \left(1-\frac{1}{2^{1+1/N}}\right)2^{\Lmax+1}2^{\Lmax/N} + \frac{2}{2^{1+1/N}-1} - 2^{1+\Lmax/N}\right\}\\
 &\le \mathrm{e}^{N-1} \left(1+ \frac{\Lmax}{N-1}\right)^{N-1 } \left(2^{1/N}-1\right)2^{\Lmax+2}2^{\Lmax/N}.
\end{align*}
Combining the last two inequalities, along with \eqref{Mestimate}, we get
\begin{align*}
    K_\mathrm{acc}(\varepsilon) & \le   \mathrm{e}^{N-1} \left(1+ \frac{\Lmax}{N-1}\right)^{N-1 }2^{\Lmax+1} \\
                   & \qquad \times  \frac{1}{ \log\left( \frac{\sqrt{\overline{\kappa}} +1}{\sqrt{\overline{\kappa}}-1}\right)} \log\left(4\sqrt{\frac{\alpha}{\beta}}\right) + 2N\log\left(\Lmax+2\right) + 2rN\left(2^{1/N}-1\right)2^{\Lmax/N}
\end{align*}
Substituting \eqref{lem:hofepsilon} for $\Lmax$ concludes the proof.
\end{proof}


In the case of the accelerated SC method, an interpolant $\mcI_{\LLL-1} [\widetilde{u}_h]$, defined by \eqref{eq:lagrangeinterp} and \eqref{eq:solution}, must be evaluated for each of the $\Delta M_\LLL$ collocation points in $\Delta \mcH_\LLL$. Each interpolant evaluation costs $2M_{\LLL-1}-1$ operations, i.e.,~additions and multiplications, and must be evaluated for each of the $M_h$ components of the FEM coefficient vector. Then the interpolation cost on each level is $M_h\Delta M_\LLL (2M_{\LLL-1}-1)$ for $\LLL = 1, \ldots, \Lmax(\varepsilon)$. Now we give an estimate of the total interpolation cost $\mathcal{C}_{\text{int}}(\varepsilon)$ for our algorithm to achieve the prescribed accuracy $\varepsilon$. 

\begin{theorem}\label{thm:interpcosts}
Given Assumption \ref{assum:reg} and the conditions of Lemma \ref{assum:FEMconv}, 
 for sufficiently small $\varepsilon > 0$, the total cost of interpolation when using the sparse grid  interpolation method in \eqref{c0jcwyj} is bounded by
 \begin{align*}
\mathcal{C}_\mathrm{int}(\varepsilon) \le 
M_{h}C_8 \left( \frac{1}{rN}\log\left(\frac{3C_\text{\em sc}}{\varepsilon}\right)\right)^{2N}
 \left\{C_2+ \frac{1}{\log{2}} \log\log\left( \frac{3 C_{\text{\em sc}}}{\varepsilon} \right) \right\}^{2(N-1)},
\end{align*}
where $C_2$ are defined as in Theorem \ref{thm:NCGzero}, and $C_8 = 64\, \mathrm{e}^{2(N-1)}$.

\end{theorem}

\begin{proof} 
The total interpolation cost is bounded by 
\begin{align}
\mathcal{C}_\mathrm{int}(\varepsilon) & \le 2M_h \sum_{\LLL=2}^{\Lmax(\varepsilon)} \Delta M_\LLL M_{\LLL-1}\notag\\
&\le 2M_h  \sum_{\LLL=2}^{\Lmax(\varepsilon)} 2^\LLL \binom{N-1+\LLL}{N-1} \sum_{l=1}^\LLL 2^l \binom{N-1+l}{N-1}\notag\\
&\le 2M_h \sum_{\LLL=2}^{\Lmax(\varepsilon)} 2^\LLL \binom{N-1+\LLL}{N-1} \sum_{l=1}^\LLL 2^l \binom{N-1+l}{N-1}\notag\\
&\le 2M_h \sum_{\LLL=2}^{\Lmax(\varepsilon)} 2^\LLL \left\{\binom{N-1+\LLL}{N-1}\right\}^2 2^{\LLL+1}\notag\\
&\le 4M_h \left\{\binom{N-1+\Lmax(\varepsilon)}{N-1}\right\}^2 4^{\Lmax(\varepsilon)+1} \notag\\
& \le 16M_h \mathrm{e}^{2(N-1)} 4^{\Lmax(\varepsilon)} \left( 1 + \frac{\Lmax(\varepsilon)}{N-1}\right)^{2(N-1)}. \label{int_cost_bound}
\end{align}
Substituting the definition of $\Lmax(\varepsilon)$ from Lemma \ref{lem:hLestimates} into \eqref{int_cost_bound} concludes the proof.
\end{proof}

Based on Theorems \ref{thm:NCGzero}, \ref{thm:NCGaccel} and \ref{thm:interpcosts}, we finally discuss the savings of the accelerated SC method proposed in \S \ref{sec:algo}. By comparing the estimates of $K_{\text{zero}}(\varepsilon)$ 
 and $K_{\text{acc}}(\varepsilon)$, 
 we see that the acceleration technique reduces $\log(C_3 / \varepsilon)$ in \eqref{Kzero} to $2\left(2^{1/N}-1\right)\log\left({3C_\text{sc}}/{\varepsilon}\right)$ in \eqref{Kacc}. Here both terms are of the same asymptotic order with respect to $\varepsilon$, but the savings from acceleration increases with dimension $N$ since $(2^{1/N}-1) \rightarrow 0$ as $N \rightarrow \infty$. On the other hand, when taking into account the cost of interpolation  $\mathcal{C}_{\text{int}}$, we must consider the cost $\mathcal{C}_{\text{iter}}$ of performing each iteration. In the case of using CG solvers, $\mathcal{C}_{\text{iter}}$ is the cost of one matrix-vector multiplication, and will be determined by the size of the unknown vector, $M_h$, and the sparsity of the mass matrix $\bm A(\bm y)$. Thus $\mathcal{C}_{\text{iter}}$ is proportional to the size of the finite element vector, i.e., $\mathcal{C}_{\text{iter}} = C_D M_h$, where $C_D$  depends on the dimension $d$ of the physical domain and choice of finite element basis. For example, without the use of a preconditioner, we can assume that the condition numbers of the matrices $\bm A(\bm y)$, for $\bm y \in \Gamma $, satisfy
\begin{align*}
\overline{\kappa} := \sup_{\bm y \in\Gamma} \kappa(\bm y) \le \left(\frac{C_{\kappa}}{h}\right)^2,
\end{align*}
where the constant $C_{\kappa}>0$ is independent of $\bm y\in\Gamma$
\cite{bank1989conditioning}. 
Then we can examine the contribution of the condition number in Theorems \ref{thm:NCGzero} and \ref{thm:NCGaccel}:
using the inequality $\log(x) \ge (x-1)/x$ and Lemmas \ref{assum:FEMconv} and \ref{lem:hLestimates}, we bound the terms involving the condition number as
\begin{align*}
\frac{1}{\log\left(\frac{\sqrt{\overline{\kappa}}+1}{\sqrt{\overline{\kappa}}-1}\right)} & \le \frac{\sqrt{\overline{\kappa}}+1}{2} \le C_{\kappa}\left(\frac{3C_{\text{fem}}}{\varepsilon}\right)^{1/s}.
\end{align*}
Now as $\varepsilon\rightarrow 0$, the asymptotic iterative solver costs, 
$\mathcal{C}_\mathrm{zero} = C_D M_h K_\mathrm{zero}$ are of the order 
$M_h\left(\frac{1}{\varepsilon}\right)^{1/s}\left\{\log\left( \frac{1}{\varepsilon}\right)\right\}^{N+1} \left\{\log\log\left(\frac{1}{\varepsilon}\right)\right\}^{N-1}$, while in the accelerated case, the estimate for $C_D M_h K_\mathrm{acc}$, is of the same order with respect to $\varepsilon$, but with an improvement to the constant 
 of $\left(2^{1/N}-1\right)$. For the accelerated method, the additional interpolation costs $\mathcal{C}_\mathrm{int}$ are of order 
$M_h \left\{ \log\left(\frac{1}{\varepsilon}\right)\right\}^{2N}
 \left\{\log\log\left( \frac{1}{\varepsilon} \right) \right\}^{2(N-1)} $,
which is 
 negligible compared to the iterative solver complexity. 
 It is clear that, asymptotically, the accelerated method leads to a net reduction in computational cost. 
We remark that for many adaptive interpolation methods, the addition of new points already involves evaluation of the current (coarse) interpolant. In this case, the cost of interpolation can be ignored, and the accelerated method should be used.

\section{Numerical examples}\label{sec:numerical}

The goal of this section is to demonstrate the reduction in computational cost of SC methods using the proposed acceleration technique. 
In Example 5.1, we first use the accelerated SC method to solve an stochastic elliptic PDE with one spatial dimension, and compute the overall cost and iteration savings gained by acceleration. Example 5.2 considers a similar problem and looks at the number of CG iterations versus the collocation error, comparing the implementation of the method using isotropic and anisotropic sparse grids, and demonstrating the effect of varying stochastic dimension $N$ on the convergence of the individual systems. In addition, as described in Remark~\ref{rem:preconditioner}, we extend our acceleration technique to interpolated preconditioners, which also exhibit the convergence improvements of the method. Finally, Example 5.3 applies the accelerated method to iterative solvers for nonlinear parametrized PDEs.  

The analysis in section \ref{sec:estimate} consisted of two components: (i) estimates for the reduction in solver iterations from using acceleration, and (ii) interpolation costs. The interpolation costs can be computed exactly for non-adaptive methods, and for adaptive implementations of sparse grid SC the interpolation costs can be ignored. In Example 5.1, all error contributions are balanced, and the total cost is examined, including both solver iterations and interpolation construction. In Examples 5.2 and 5.3 we focus only on the number of iterations of the CG solver.

\subsection*{Example 5.1}\label{ssec:1d}
We consider the following elliptic stochastic PDE
\begin{equation}\label{prob1}
\left\{
\begin{array}{rll}
-\nabla \cdot \left( a\left(x,\bm y\right) \nabla u\left(x,\bm y\right) \right) &= ~10 & \textrm{ in } D \times \Gamma , \\
u(x,\bm y) &= ~0 & \textrm{ on } \partial D \times \Gamma,
\end{array}
\right.
\end{equation}
where $D=[0,1]$, $\bm y = (y_1, y_2, y_3, y_4)^{\top}$, $\Gamma_n = [-1,1], n=1,\ldots,4$, and the coefficient $a$ is given by:
\begin{align}\label{eq:coeff1d}
    \log\left(a\left(x,\bm y\right)-1\right)  &= \mathrm{e}^{-1/8}\left(y_1 \cos \pi x  + y_2 \sin \pi x  + y_3 \cos 2\pi x  +  y_4 \sin 2\pi x \right).
\end{align}
The random variables $\{y_i\}_{i = 1}^4$ are independent and identically distributed uniform random variables in $[-1,1]$. 
In the one-dimensional physical domain, a finite element discretization using linear elements yields tridiagonal, symmetric positive-definite systems. While this type of system could be solved efficiently by direct methods, nevertheless we use CG solvers to demonstrate the convergence properties of the acceleration method.

Table \ref{tbl:NOP} compares the standard and the accelerated SC methods, where the error for each
approximate solution,~$\widetilde{u}_{h,\Lmax}$, is computed against a highly refined approximate reference solution $\widetilde{u}_{h^*,\LLL^*}$
with $h^*=2^{-14},\LLL^*=10$. 
In Figure \ref{fig:1Dsav} we plot the savings of the accelerated SC method, computed according to the cost metrics \eqref{eq:cost0} and \eqref{eq:costa}. Since the constants $C_{\text{fem}}$ and $C_{\text{sc}}$ in Lemma \ref{lem:hLestimates} are not known {\em a priori}, to balance the error contributions in \eqref{e1e2e3} we 
use trial and error to determine sufficient values $h$, $\Lmax$, and $\tau$ to achieve the desired overall error $\varepsilon$ in the $L^2_\varrho$ norm. Especially for the larger systems, i.e., those with a large number of spatial degrees of freedom, significant savings are achieved.
The percent savings in the number of iterations versus the cost of interpolation are calculated according to
\begin{equation*}
	\frac{\mathcal{C}_\mathrm{zero} - \mathcal{C}_\mathrm{acc}}{\mathcal{C}_\mathrm{zero}} = \frac{M_h C_D (K_\mathrm{zero} - K_\mathrm{acc} ) - \mathcal{C}_\mathrm{int} }{M_h C_D K_\mathrm{zero} },
\end{equation*}
where $C_D=5$, since the matrices are tridiagonal. 
\begin{table}[h!]
\renewcommand{\arraystretch}{1.2}
\begin{center}
\begin{tabular}{|c | c | c | c | c | c | c |}
\hline
Tot. Err	&FE DoFs & SC Pts	& CG tol & $K_\mathrm{zero}$ & 	$K_\mathrm{acc}$ 	& Savings \\ \hline 
 $1\times 10^{-2}$&	255	&	137		& $1\times10^{-3}$	&	28,259   		&	21,123		&	19.4 \%	\\	
 $5\times 10^{-3}$&	511	&	401		& $5\times10^{-3}$	&	173,671		&	83,884		&	42.4\%	\\
 $1\times 10^{-3}$&	2,047	&	1,105		& $1\times10^{-4}$	&	2,001,905		&	626,215		&	62.3\% \\
 $5\times 10^{-4}$&	4,095	&	2,929		& $5\times10^{-5}$	&	10,878,352	&	1,842,703		&	74.5\% \\
 $1\times 10^{-4}$&	16,383&	7,537		& $1\times10^{-5}$	&	114,570,175	&	12,345,968	&	75.1\% \\
 \hline
\end{tabular}
\end{center}
\caption{Comparison in computational cost between the standard and the accelerated SC methods for solving \eqref{prob1}--\eqref{eq:coeff1d}.}
 \label{tbl:NOP}
\end{table}
\begin{figure}[h!]
\begin{center}
	\includegraphics[width=9cm]{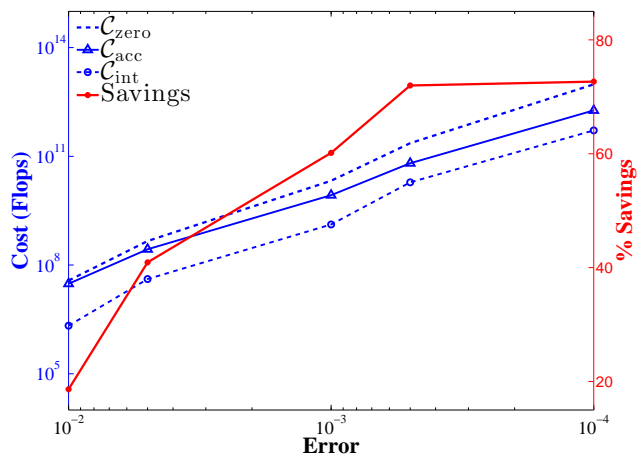}
\end{center}
\caption{Cost (left axis) and percent savings (right axis) of the accelerated SC method versus the standard SC method for solving \eqref{prob1}--\eqref{eq:coeff1d}. Costs are computed according to \eqref{eq:cost0} and \eqref{eq:costa}.}
\label{fig:1Dsav}
\end{figure}

\subsection*{Example 5.2}
We consider the following stochastic linear elliptic problem
\begin{equation}\label{numeqn}
\left\{
\begin{alignedat}{2}
-\nabla \cdot \left( a\left(x,\bm y\right) \nabla u\left(x,\bm y\right) \right) & = \cos(x_1)\sin(x_2) & \quad \textrm{~ in ~} D\times\Gamma , \\
u(x,\bm y) & = 0 & \quad \textrm{~ on ~} \partial D \times \Gamma,
\end{alignedat}
\right.
\end{equation}
where $D = [0,1]\times[0,1]$, $\Gamma_n=[-\sqrt{3},\sqrt{3}], n=1,\ldots,N$, and $x= (x_1,x_2)$ is the spatial variable. The random diffusion term has one-dimensional spatial dependence given by
\begin{subequations}\label{numsetup}
\begin{align}\label{aexpansion}
\log(a(x,\bm y) -0.5) &= 1 + y_1 \left(\sqrt{\pi}R/2\right)^{1/2} + \sum_{n=2}^N\zeta_n \varphi_n(x) y_n,
\intertext{where}
\label{zeta}
\zeta_n &:= (\sqrt{\pi~}R)^{1/2}\exp\left(\frac{-\left(\left\lfloor {n}/{2}\right\rfloor\pi R\right)^2}{8}\right), \qquad n>1
\intertext{and}
\label{varphi}
    \varphi_n(x) & :=
    \begin{cases}
        \vspace{0.2cm}\sin\left(\dfrac{\left\lfloor n/2\right\rfloor\pi x_1}{R_p}\right), \quad n \text{ even}\\
        \cos\left(\dfrac{\left\lfloor n/2\right\rfloor\pi x_1}{R_p}\right),  \quad n \text{ odd.}
    \end{cases}
\end{align}
\end{subequations}
The random variables $\{y_n\}_{n=1}^N$ are~i.i.d.~and are each uniformly distributed in $[-\sqrt{3},\sqrt{3}]$, with zero mean and unit variance, i.e., $\E[y_n]=0$, and $\E[y_ny_m]=\delta_{nm}$, for $n,m\in \N_+$. The finite dimensional stochastic diffusion $a$ represents the $N$-term truncation of an expansion of a random field with stationary covariance function, given by
\begin{equation}\label{eq:cov}
  \text{Cov}\left[\log\left(a - 0.5\right)\right](x_1,x_1') = \exp\left( -\frac{(x_1-x_1')^2}{R_c^2} \right),
\end{equation}
where $x_1,x_1'\in [0,1]$, and $R_c$ is the physical correlation length for the random field $a$. The parameter $R_p$ in (\ref{varphi}) is given by $R_p = \max\{1,2R_c\}$ and $R$ is given by $R= R_c/R_p$. Then $\zeta_n$ and $\varphi_n(x)$ are the eigenvalues and eigenfunctions associated with \eqref{eq:cov}.
Here we will consider two correlation lengths, namely $R_c=1/2$, and $R_c=1/64$, where Figure \ref{fig:eigdecay} shows the corresponding decay of eigenvalues.
For the spatial discretization, we use a finite element approximation on a regular triangular mesh with linear finite elements and 4225 degrees of freedom. 
The CG method is used for the linear solver with diagonal preconditioners and a tolerance of $10^{-14}$. 
\begin{figure}
  \begin{center}
      \includegraphics[width=7cm]{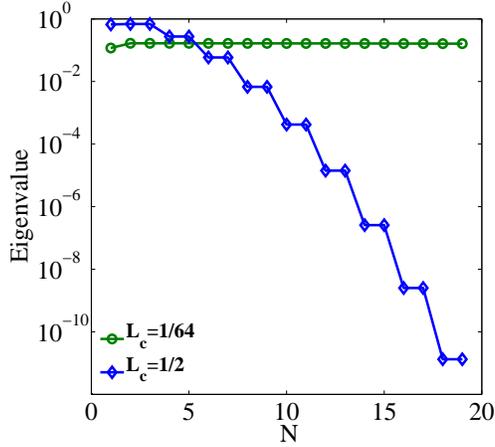}
  \end{center}
  \caption{First 19 eigenvalues for \eqref{eq:cov} for correlation length $R_c=1/64, 1/2$.}
\label{fig:eigdecay}
\end{figure}

First, for $R_c=1/64$, the error and total iteration count of both the standard case, using zero initial vectors, and accelerated SC construction, computed using several dimensions $N$, are summarized in Table \ref{table:L64}.
The error is measured using the expectation of the approximate solutions, $\|\E[u_{h,\Lmax}] - \E[u_{h,\LLL^*}]\|_{L^2(D)}$, for $\Lmax=1,\ldots,7$, where the ``exact" solution $\E[u_{h,\LLL^*}]$ is computed using $\LLL^*=8$. We compare these errors against the cumulative total number of iterations, $K_\mathrm{zero}$ and $K_\mathrm{acc}$, needed to construct $\E[u_{h,\Lmax}]$. 

\begin{table}
\renewcommand{\arraystretch}{1.1}
\begin{center}
\begin{tabular}{|c|c|c|c|c|c|}
\cline{1-6}
{} & {Error} & {SC Pts} & {$K_\mathrm{zero}$ } & {$K_\mathrm{acc}$ } & {Savings in K}\\
\hline
\multicolumn{1}{ |c| }{\multirow{3}{*}{N=3}} & \multicolumn{1}{ |c| }{3.83e-8} & {25} & {6,780} &  {5,991} & {11.6\%}\\
\cline{2-6}
\multicolumn{1}{ |c| }{} & \multicolumn{1}{ |c| }{9.57e-10} & {69} &  {18,893} & {14,628} & {22.6\%}\\
\cline{2-6}
\multicolumn{1}{ |c| }{} & \multicolumn{1}{ |c| }{9.86e-12} & {177} & {48,691} &  {27,765} & {43.0\%}\\
\hline
%
\multicolumn{1}{ |c| }{\multirow{3}{*}{N=5}} & \multicolumn{1}{ |c| }{5.28e-07} & {61} & {17058} &  {15095} & {11.6\%}\\
\cline{2-6}
\multicolumn{1}{ |c| }{} & \multicolumn{1}{ |c| }{1.03e-08} & {241} &  {67,955} & {53,992} & {20.6\%}\\
\cline{2-6}
\multicolumn{1}{ |c| }{} & \multicolumn{1}{ |c| }{1.44e-10} & {801} & {226,597} &  {150,241} & {33.7\%}\\
\hline
\multicolumn{1}{ |c| }{\multirow{3}{*}{N=7}} & \multicolumn{1}{ |c| }{2.43e-08} & {589} &  {168,237} & {136,072} & {19.1\%}\\
\cline{2-6}
\multicolumn{1}{ |c| }{} & \multicolumn{1}{ |c| }{6.63e-10} & {2,465} & {706,049} &  {500,718} & {29.1\%}\\
\cline{2-6}
\multicolumn{1}{ |c| }{} & \multicolumn{1}{ |c| }{1.94e-11} & {9,017} & {2,585,970} &  {1,496,391} & {42.1\%}\\
\hline
%
\multicolumn{1}{ |c| }{\multirow{3}{*}{N=9}} & \multicolumn{1}{ |c| }{1.68e-07} & {1,177} & {338,428} & {277,583} & {18.0\%}\\
\cline{2-6}
\multicolumn{1}{ |c| }{} & \multicolumn{1}{ |c| }{7.83e-09} & {6,001} & {1,729,337} &  {1,273,895} & {26.3\%}\\
\cline{2-6}
\multicolumn{1}{ |c| }{} & \multicolumn{1}{ |c| }{8.86e-11} & {26,017} & {7,505,343} &  {4,719,820} & {37.1\%}\\
\hline
\multicolumn{1}{ |c| }{\multirow{3}{*}{N=11}} & \multicolumn{1}{ |c| }{2.59e-07} & {2,069} & {596,368} & {495,705} & {16.9\%}\\
\cline{2-6}
\multicolumn{1}{ |c| }{} & \multicolumn{1}{ |c| }{2.43e-08} & {12,497} & {3,608,185} & {2,736,615} & {24.2\%}\\
\cline{2-6}
\multicolumn{1}{ |c| }{} & \multicolumn{1}{ |c| }{1.95e-09} & {63,097} & {18,231,420} & {12,139,658} & {33.4\%}\\
\hline

\end{tabular}
\end{center}
\caption{Iteration counts and savings of the accelerated SC method for solving \eqref{numeqn}--\eqref{numsetup} with correlation length $R_c=1/64$, and stochastic dimensions $N=5,7,9$, and $11$.}
\label{table:L64}
\end{table}

An alternative approach to accelerating SC methods is found in \cite{gordon2012solving}. For a particular SC level $\Lmax$, this method orders the collocation points lexicographically, with each dimension ordered according to
the decay of the eigenvalues in (\ref{aexpansion}). 
 We also implemented a similar method without the sequential ordering; for a given level $\LLL$, at each new collocation point in $\Delta\mathcal{H}_\LLL$ the solution at the nearest collocation point from lower levels is given as an initial guess to accelerate the CG solver. We refer to this method as the ``nearest neighbor" approach. 
 Figure \ref{fig:avgiters} shows the 
 average number of iterations needed to solve the linear system \eqref{jlinsys}, where the average is taken over the new points at level $\LLL$, i.e., $\Delta\mathcal{H}_\LLL$, for $\LLL=1,\ldots,7$. We compare our interpolated acceleration algorithm, the nearest neighbor approach, and standard SC method without acceleration, for $N=3$ and $N=11$, using $R_c=1/64$. The interpolated initial vector provided by the acceleration algorithm yields a reduction in the average number of iterations at each level, which increases with $\LLL$. 
  Figure \ref{fig:avgiters} also shows the effect of using the nearest neighbor solution as the initial vector, 
 which provides some improvement over the standard case using zero initial vectors, but the savings do not match those of our acceleration scheme. Note that since the number of new collocation points grows exponentially with each level (cf \eqref{growthSmolyak}), there is an increase in {\em total} iteration savings over successive levels in both the nearest neighbor and accelerated case.
 
The left plot of Figure \ref{fig:savingsComparison} shows the total iteration savings achieved by the acceleration algorithm with different maximum collocation levels $\Lmax=1,\ldots,6$. The savings are measured as the percentage reduction in the cumulative iteration count up to level $\Lmax$, relative to standard case using zero initial vectors, i.e.,  $(K_\mathrm{zero}-K_\mathrm{zero})/K_\mathrm{zero}$. Here we also see the effect of stochastic dimension on the convergence of SC methods: as $N$ increases, our algorithm provides less accurate initial guesses for a given maximum SC level $\Lmax$. This can also be seen by comparing the left and right plots of Figure \ref{fig:avgiters}, which show how the average number of iterations at a given SC level $\LLL$ changes from $N=3$ to $N=11$.  On the other hand, the right plot of Figure \ref{fig:savingsComparison} shows the same total iteration savings now plotted versus error. As above, the error is measured as $\|\E[u_{h,\Lmax}] - \E[u_{h,\LLL^*}]\|_{L^2(D)}$, with $\LLL^*=7$. These results are in agreement with the theoretical asymptotic estimates from Theorem \ref{thm:NCGaccel}, which predict an increased savings vs error for larger dimensions.

\begin{figure}
  	\begin{center}
    		\includegraphics[width=7.5cm]{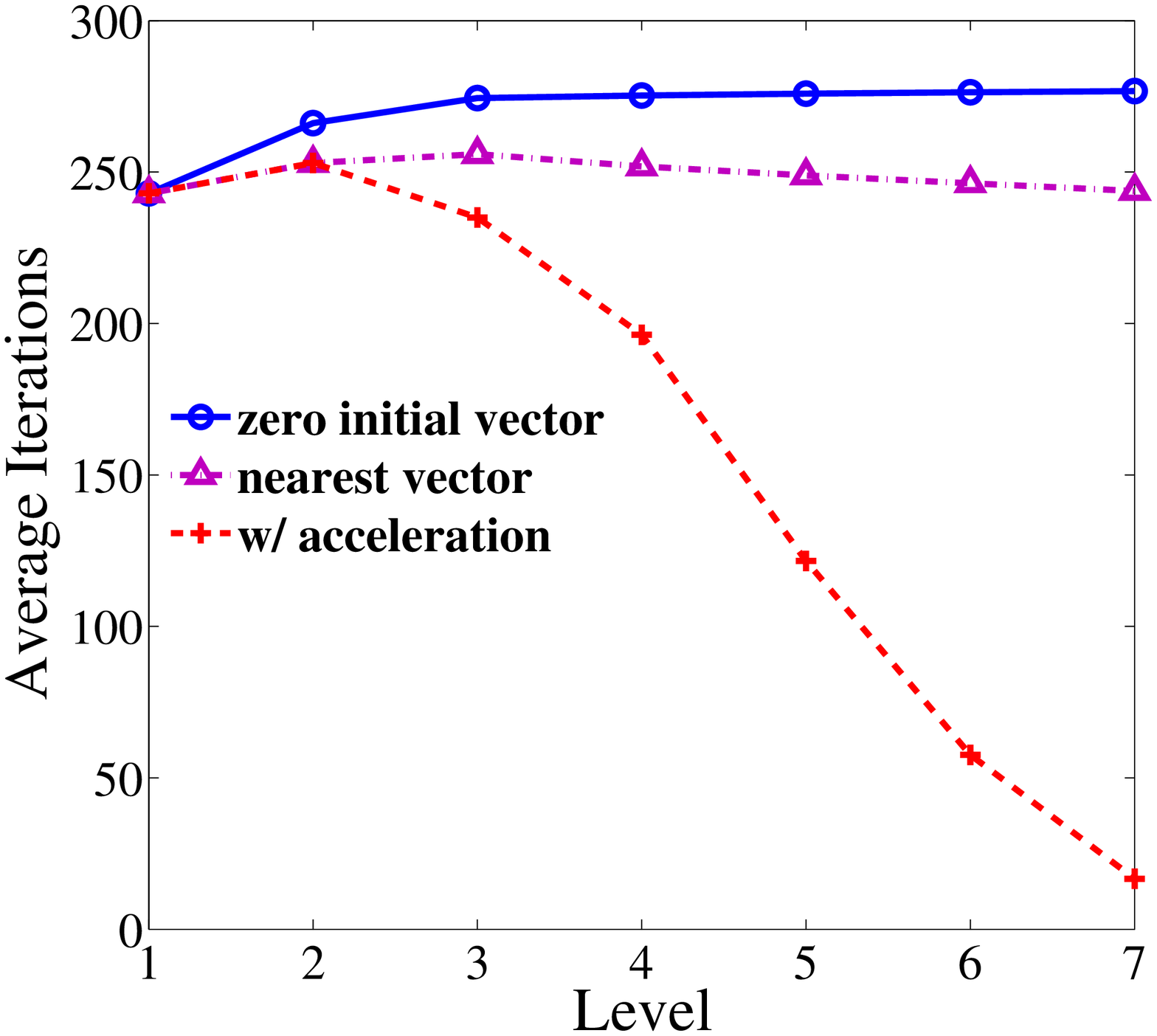}     
		\includegraphics[width=7.5cm]{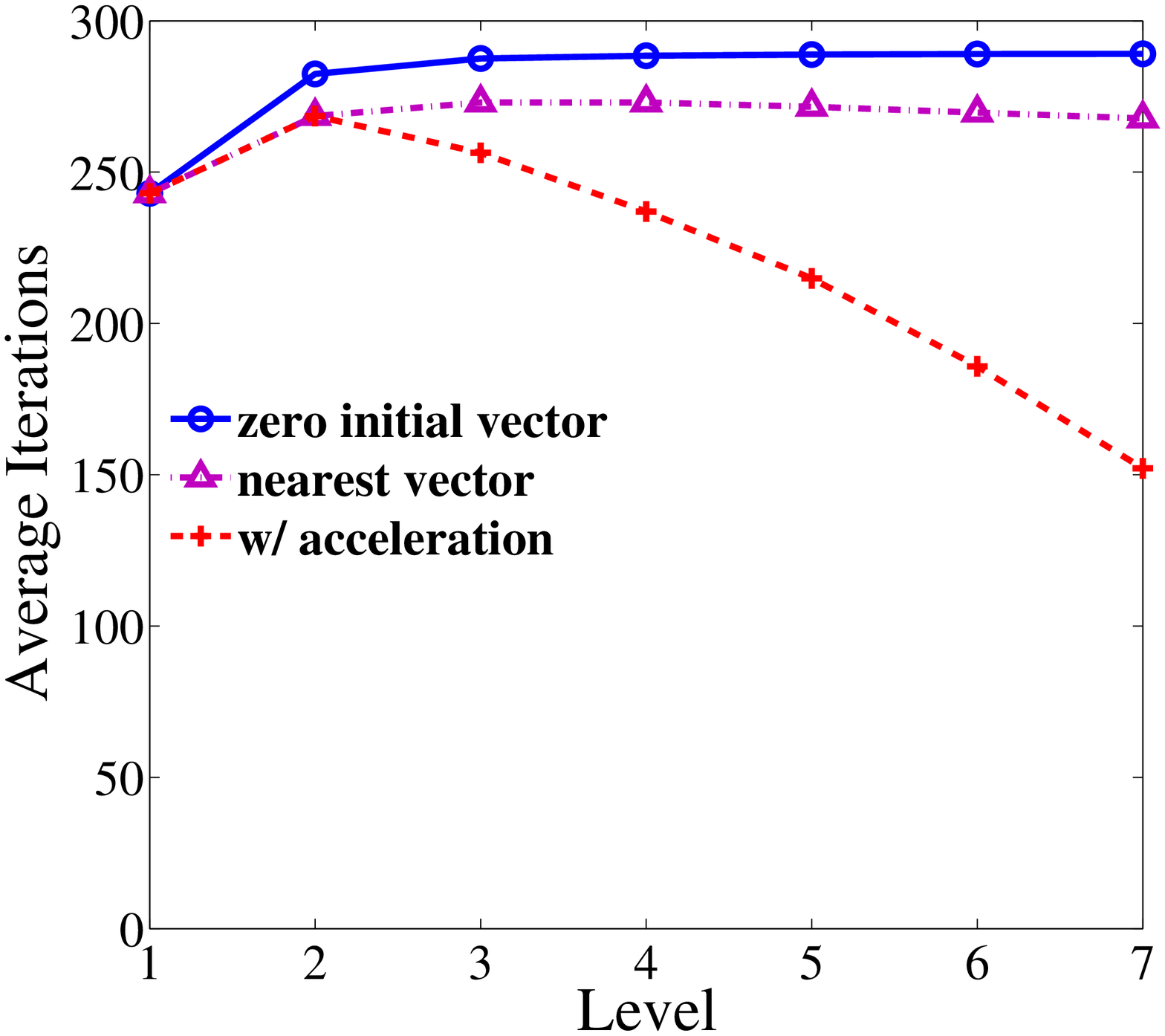}
  \end{center}
  \caption{Comparison of the average CG iterations per level for solving problem \eqref{numeqn}--\eqref{numsetup} with dimensions $N=3$ (left) and $N=11$ (right), and correlation length $R_c=1/64$.}
\label{fig:avgiters}
\end{figure}
\begin{figure}
  	\begin{center}
   		\includegraphics[width=7.5cm]{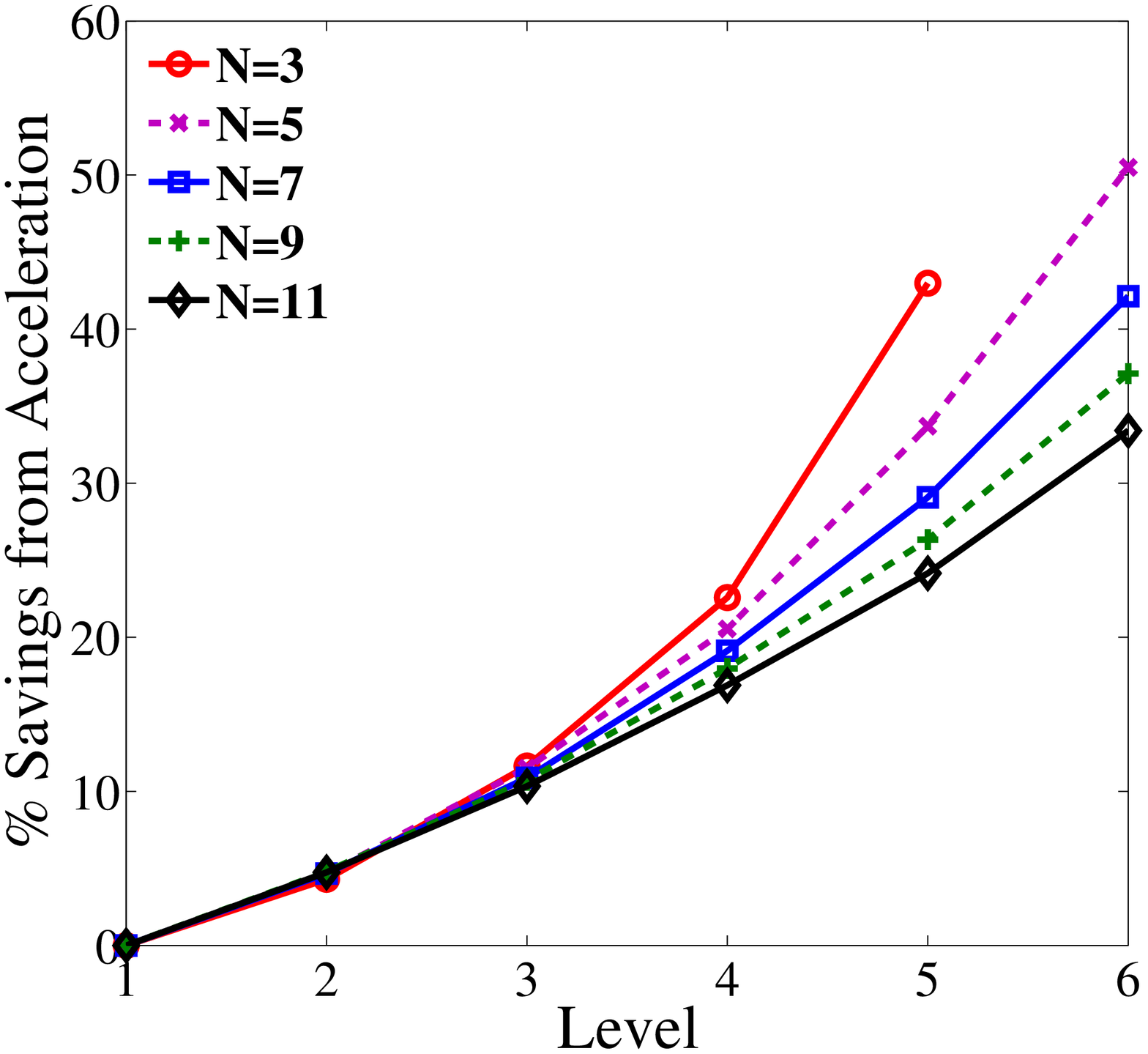}
    		\includegraphics[width=7.5cm]{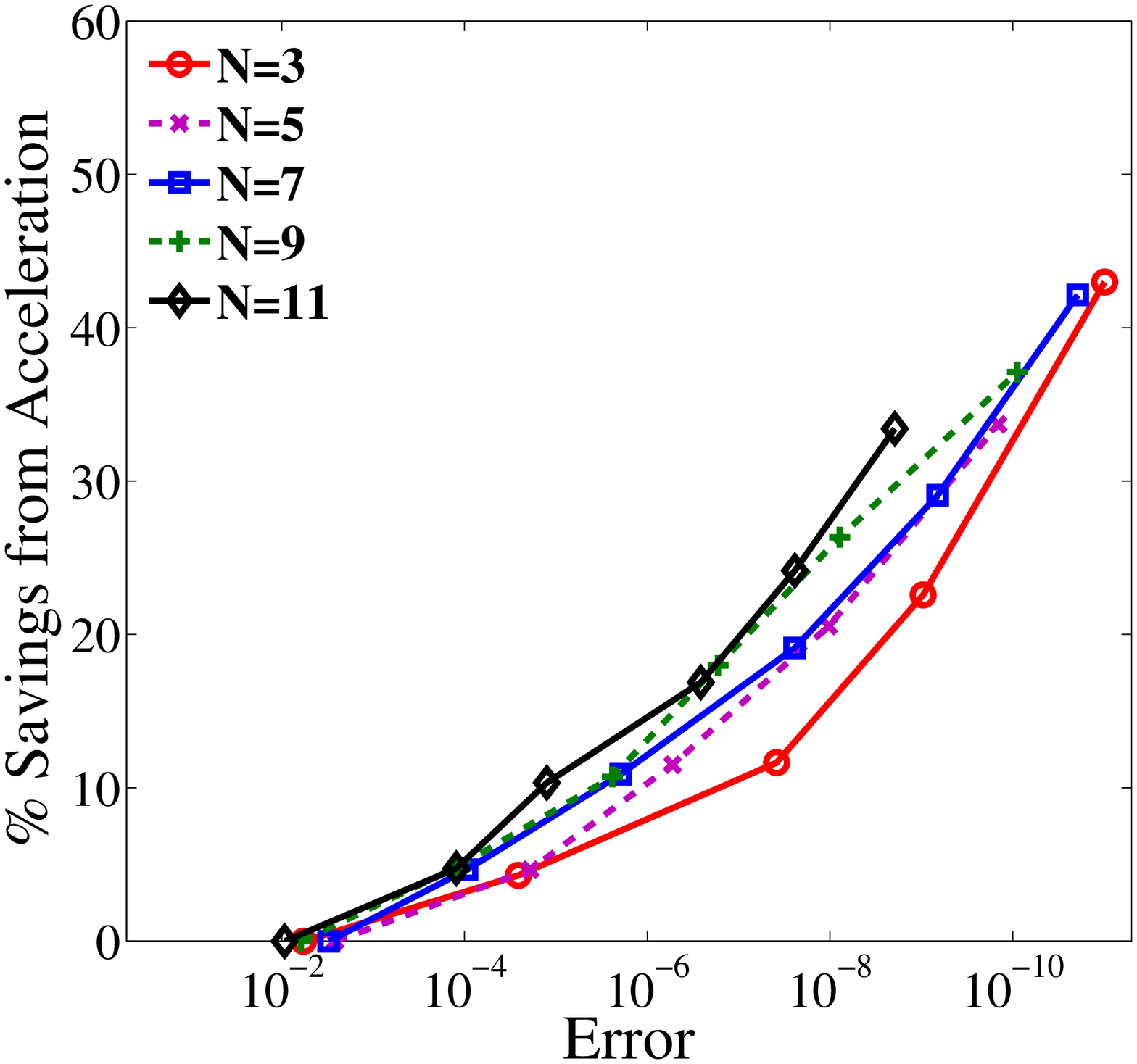} 
     	\end{center}
\caption{Percentage cumulative reduction in CG iterations vs level (left) and error (right) for solving \eqref{numeqn}--\eqref{numsetup} using our accelerated approach, with $N=5,7,9$, and $11$ and for correlation length $R_c=1/64$.}
\label{fig:savingsComparison}
\end{figure}
\begin{figure}
  	\begin{center}
       		\includegraphics[width=7.5cm]{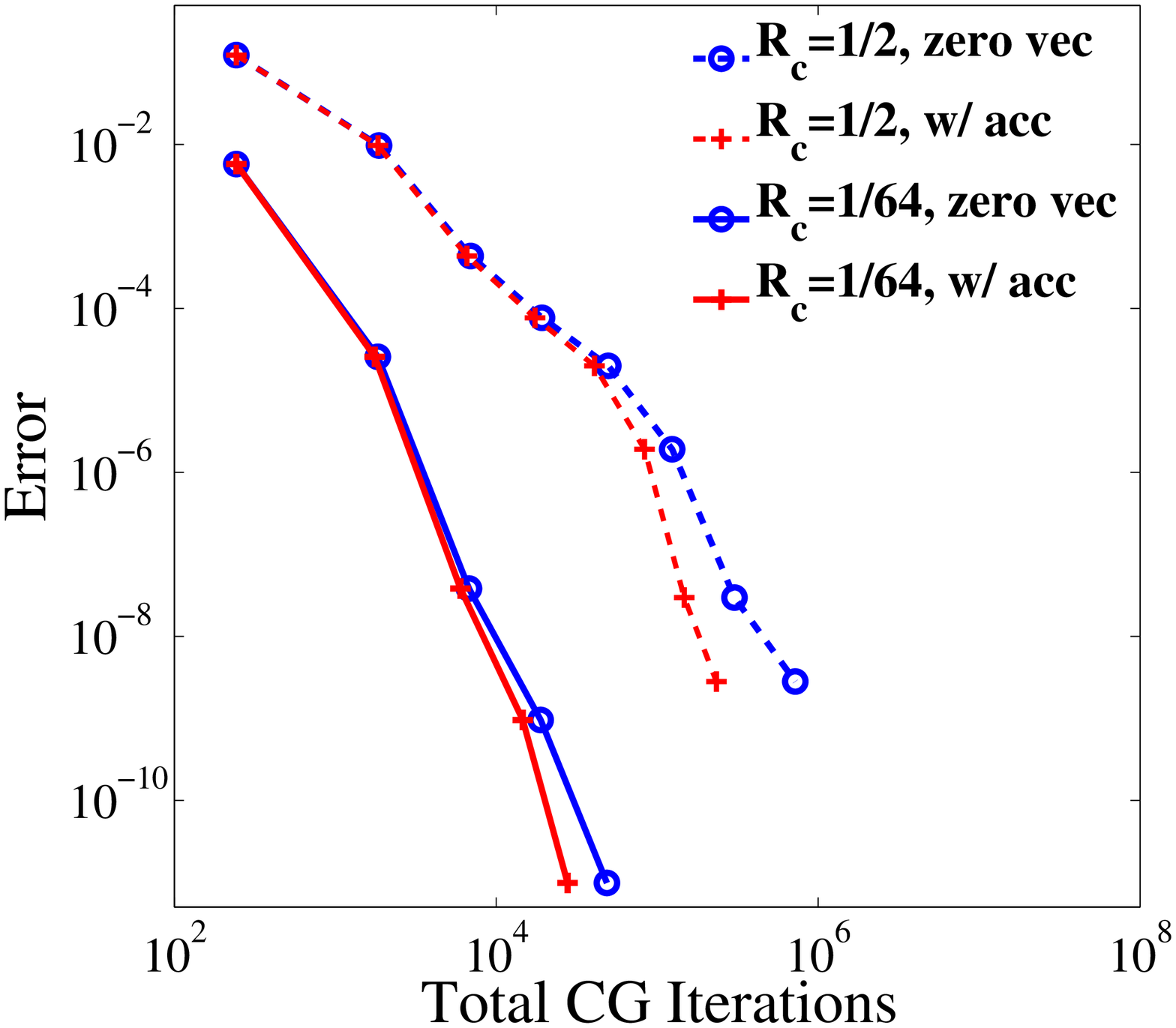}     
		\includegraphics[width=7.5cm]{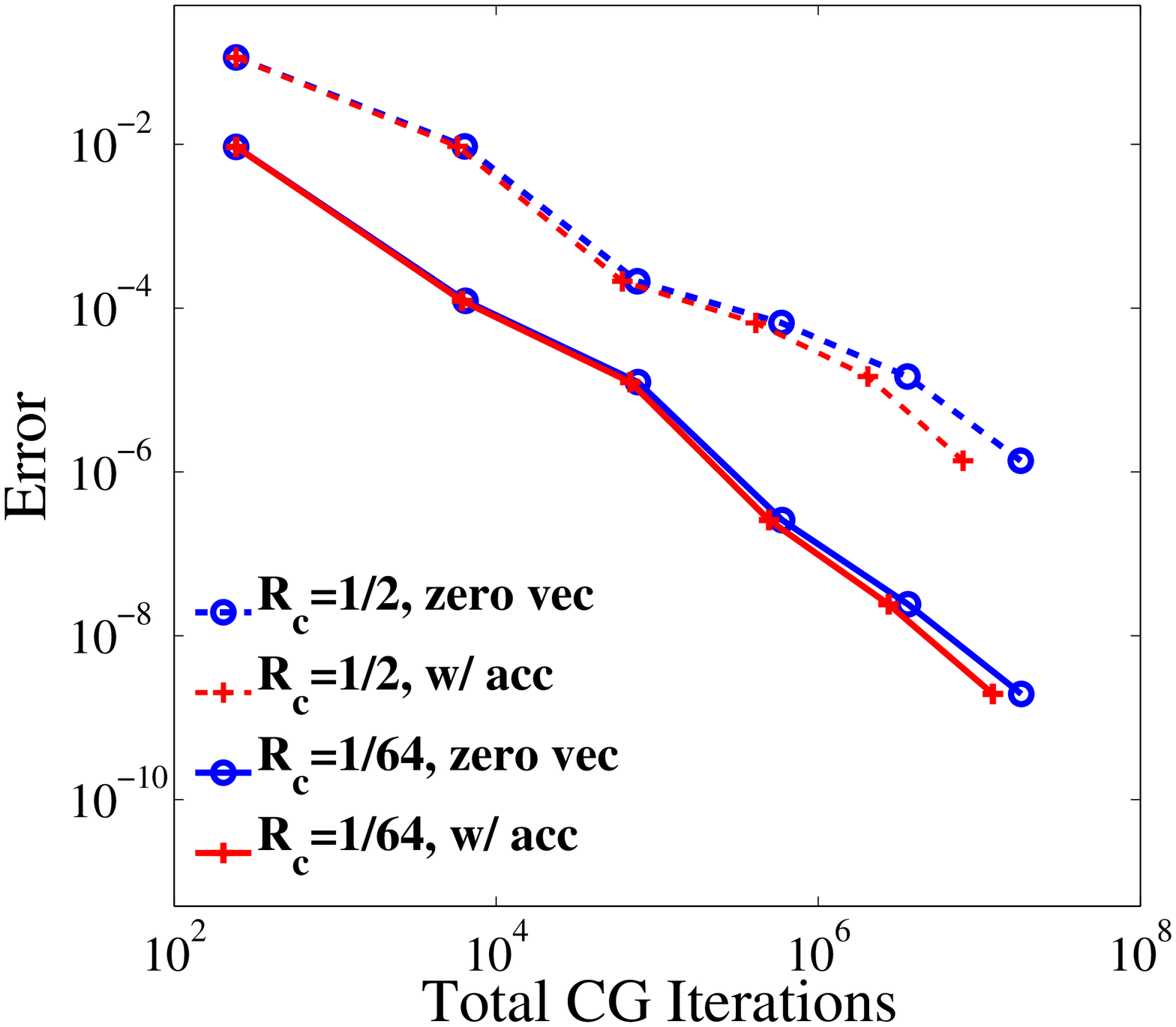}
  \end{center}
\caption{The convergence of the SC approximation for solving \eqref{numeqn}--\eqref{numsetup}, using CG, with and without acceleration, for correlation lengths $R_c=1/64, 1/2$, and dimensions $N=3$ (left), and $N=11$ (right). }
\label{fig:L2vsL64}
\end{figure}

Next we examine the effect of the correlation length, $R_c$, on our acceleration algorithm. Larger correlation lengths result in faster decay of eigenvalues of the covariance function \eqref{eq:cov} (see Figure \ref{fig:eigdecay}), and implies that $u(\bm y)$ depends on certain components of the vector $\bm y$ more than others, which reduces the effectiveness of isotropic methods.  
Figure \ref{fig:L2vsL64} plots the convergence of the error in $\E[u_{h,L}]$ versus the total number of CG iterations for $N=3$ and $N=11$, and for both $R_c=1/2$ and $R_c=1/64$.  The larger correlation length, $R_c=1/2$, results in slower convergence of the SC interpolant than for $R_c=1/64$, but note that the accelerated method reduces the total iteration count in both cases.

\begin{figure}
  \begin{center}
   \includegraphics[width=7.5cm]{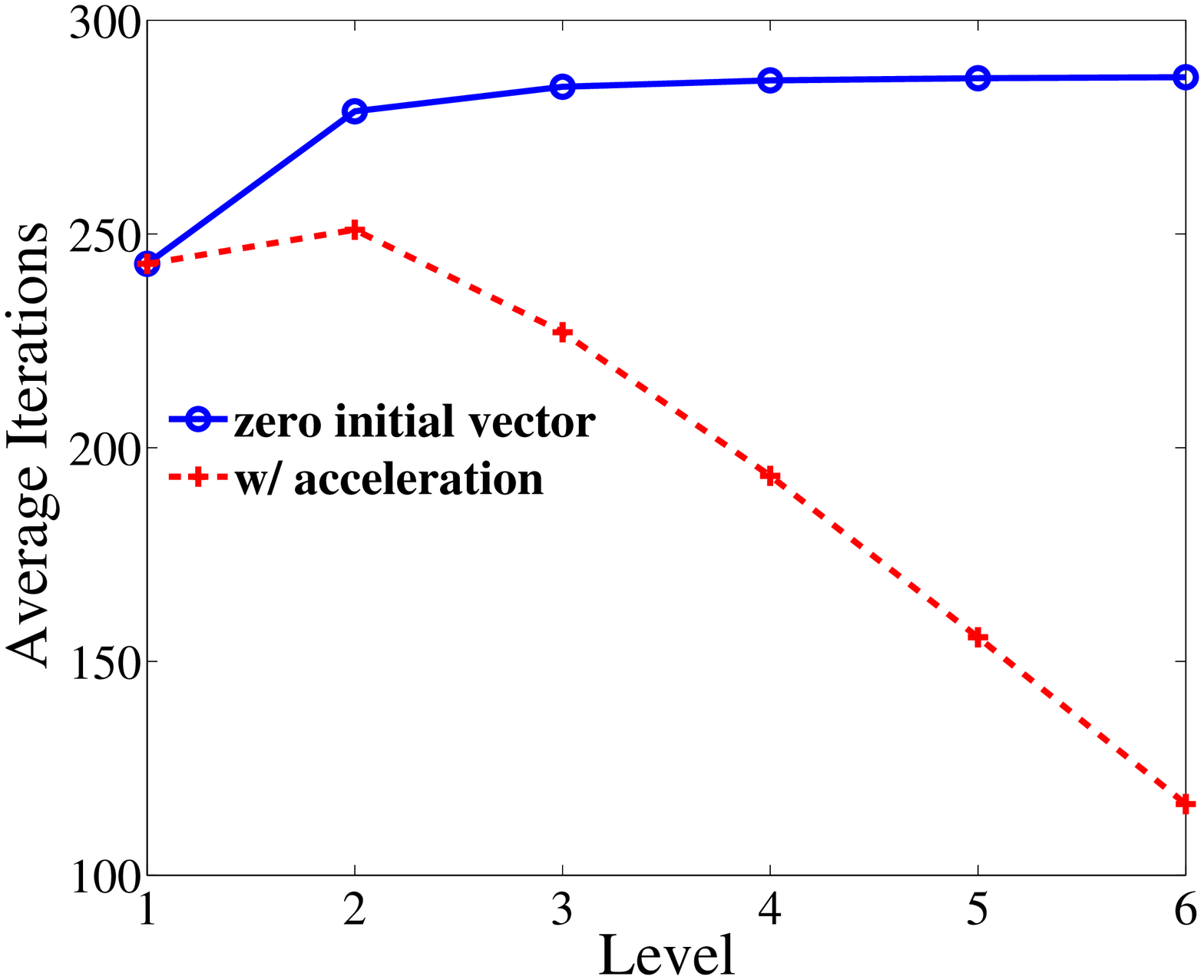}
   \includegraphics[width=7.5cm]{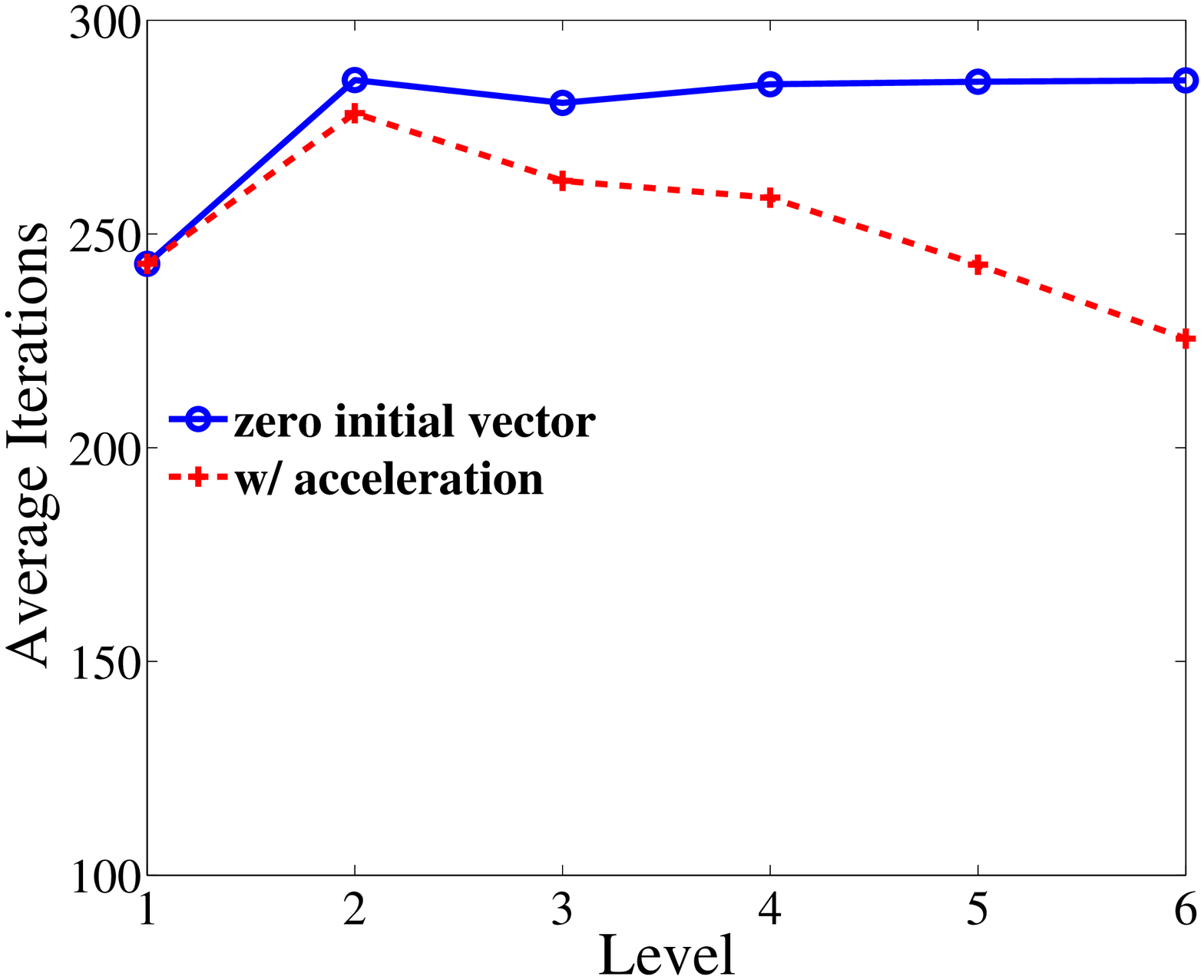}
  \end{center}
  \caption{Average CG iterations per level for solving problem \eqref{numeqn}--\eqref{numsetup} for $N=11$ and with correlation length $R_c=1/2$, using an isotropic SC (left) and anisotropic SC (right). The inefficiencies from using an isotropic grid are partially offset by increased gains from acceleration.}
\label{fig:aniso11Davg}
\end{figure}

On the other hand, we can employ {\em anisotropic} methods to increase the efficiency of SC in the case of larger correlation lengths \cite{nobile2008anisotropic}. Anisotropic SC methods will place more points in directions corresponding to large eigenvalues of \eqref{eq:cov}, and the importance of each dimension is encoded in a weight vector (see \eqref{growthSmolyak}).  Figure \ref{fig:aniso11Davg} plots the average number of iterations for problem \eqref{numeqn}--\eqref{numsetup} with a relatively large correlation length $R_c=1/2$, and $N=11$. Here we employ the weights given by an {\it a posteriori} selection described in \cite{nobile2008anisotropic}, i.e., the weight vector $\bm \alpha\in\mathbb{R}^N$, with $\alpha_1 = 0.85, \alpha_2 = \alpha_3 = 0.8, \alpha_4=\alpha_5= 1.0, \alpha_6=\alpha_7= 1.6, \alpha_8=\alpha_9 = 2.6, \alpha_{10}=\alpha_{11}= 3.7$.
The acceleration method decreases the average number of iterations needed to solve the linear system, but the effect is not as pronounced as in the case of an isotropic SC method. This occurs because the isotropic method places far too many points in relatively unimportant directions, thus the dependence of $u(\bm y)$ on a certain component $y_n$ of $\bm y$ may be well approximated at very low levels. Anisotropic methods exhibit better convergence with respect to $M_{\Lmax}$ (and lower interpolation costs) versus isotropic methods, yet we see here that the acceleration algorithm helps to somewhat offset the inefficiency of isotropic methods for anisotropic problems.

\begin{table}
\begin{center}
\begin{tabular}{|c|c|c|c|c|c|c|}
\hline
 \multicolumn{7}{|c|} {CG iterations for standard SC}\\
\hline
{Level} & {\;No PC\;} &  {\;Diag PC\;} & {\;Inc. Chol.\;} & {\;$\LLL_\mathrm{PC}=1$\;} & {\;$\LLL_\mathrm{PC}=2$\;} & {\;$\LLL_\mathrm{PC}=3$\;} \\
\hline
1 & 243 & 243 & 55 & 55 & -- & --\\
\hline
2 & 311.8 & 278.4 & 54.7 & 60.7 & 54.7 & -- \\
\hline
3 & 332.3 & 284.9 & 54.6 & 63.5 & 54.9 & 54.6 \\
\hline
4 & 341.0 & 286.1 & 54.6 & 65.2 & 55.3 & 54.6 \\
\hline
5 & 345.8 & 286.7 & 54.6 & 66.2  & 55.5 & 54.6 \\
\hline
6 & 348.4 & 286.9 & 54.6 & 66.7 & 55.6 & 54.6 \\
\hline
\hline
 \multicolumn{7}{|c|} {CG iterations for accelerated SC} \\
 \hline
{Level} & {No PC} &  {Diag PC} & {Inc. Chol.} & {$\LLL_\mathrm{PC}=1$} & {$\LLL_\mathrm{PC}=2$} & {$\LLL_\mathrm{PC}=3$} \\
\hline
1 & 243 &    243    &   55  &      55   &     --   &     --\\
\hline
2 & 299.3 &  264.6 & 52.9 & 58.4 &  52.9 &  --  \\
\hline
3 & 295.8 &  251.3  &  49.1&  57.1 &  49.4  & 49.1\\
\hline
4 & 270.8   &  225.8&  43.7  & 52.3 &  44.2 &   43.7 \\
\hline
5 & 237.0 &  194.3 &  37.3 & 45.8 &  38.0 &  37.3 \\
\hline
6 & 186.1  & 151.9  & 28.9 & 36.0 &  29.5 &  28.9  \\
\hline
\end{tabular}
\end{center}
\caption{Average iteration counts for the standard SC method (top), and the accelerated SC method (bottom)
using six preconditioner schemes to solve \eqref{numeqn}--\eqref{numsetup} with $N=7$, and $R_c=1/64$. From left to right: no preconditioner, diagonal preconditioners, incomplete Cholesky preconditioners, and accelerated preconditioners \eqref{PCalg} built using incomplete Cholesky preconditioners with $\LLL_\mathrm{PC}=1,2,3$.}
\label{table:7DPC}
\end{table}

In the preceding results we have used a simple diagonal preconditioner strategy. As described in Remark~\ref{rem:preconditioner}, we can also construct efficient preconditioners with our acceleration scheme.  Table \ref{table:7DPC} shows the effectiveness of the preconditioning strategy for solving equations \eqref{numeqn}--\eqref{numsetup}, with $N=7$ and $R_c=1/64$, where we compare the average number of iterations needed to solve \eqref{jlinsys} at each new point $\bm y_{\jL{\LLL}}\in\Delta\mathcal{H}_\LLL$ at a given level $\LLL$. Here we compute an incomplete Cholesky preconditioner for each linear system on the levels $\LLL =1,\ldots, \LLL_\mathrm{PC}$, for $\LLL_\mathrm{PC} = 1,2$, and $3$, and use these to provide an ``accelerated" preconditioner \eqref{PCalg} for the systems on the remaining levels $\LLL_\mathrm{PC}+1,\ldots, \Lmax$. We compare this against the cases where a simple diagonal preconditioner and an incomplete Cholesky preconditioner are used for each system.
The three-level accelerated preconditioner reduces the average number of iterations to within a decimal point of the incomplete Cholesky preconditioner, and the cost of computing the low-level preconditioners and interpolating is relatively cheap in comparison.
\subsection*{Example 5.3}\label{ssec:nl1d}
The preceding experiments demonstrate the benefits of using acceleration to improve the convergence of individual iterative linear solvers. In the case of a nonlinear PDE, the possibilities for savings can be even greater than the linear cases above, since convergence of a nonlinear solver may be slow or even unattainable from a poor initial vector.
In this example, we consider the problem
\begin{equation}\label{prob:nonlinear}
\left\{
\begin{alignedat}{3} 
	-\nabla \cdot \left( a\left(x,\bm y\right) \nabla u\left(x,\bm y\right) \right) + F[u](x, \bm y) & = x \qquad \textrm{ in } D \times \Gamma , \notag \\
	u(0,\bm y) & = 0 \qquad \textrm{ in } \Gamma, \\
	u'(1,\bm y) & = 1 \qquad \textrm{ in } \Gamma,
\end{alignedat}
\right.
\end{equation}
where $a$ is given by \eqref{eq:coeff1d}, $D=[0,1]$, $\Gamma_n=[-1,1], n=1,\ldots,4$, and $F[u]$ is some nonlinear function of $u$. In what follows, we consider the nonlinear functions $F[u] = u^5$, and $F[u] = uu'$.

Nonlinear problems are typically solved with the use of iterative methods such as Picard iterations or Newton's method. We implement a combination of these methods that begins with Picard iterations, then utilizes Newton's method once the relative errors are small. For spatial discretization, we use piecewise linear finite elements on $[0,1]$ with a mesh size of $h = 1/500$, and solved the resulting systems at each iteration using exact methods. The stopping criterion for the solver is a relative tolerance of $10^{-8}$ in the $l^2$ norm.

Results for these experiments are given in Figure \ref{fig:nonlinear}.  
 For each SC level, $\LLL=1,\ldots,8,$ we plot the average number of nonlinear iterations, where the average is taken over the set of points which are new to level $\LLL$, namely $\Delta \mcH_\LLL$. 
Finally, we show the total computational time in Table \ref{tbl:nonlineartime}, for different maximum levels of stochastic approximation, measured on a workstation with 1.7GHz dual core processors and 8 GB of RAM. 
We note that in Table \ref{tbl:nonlineartime}, the size of the finite element system is fixed. Thus, as we move to higher levels of collocation, the stochastic approximation becomes relatively more expensive to compute compared to the solving the finite element systems. This is why the savings begin to decrease after level 5, even though Figure \ref{fig:nonlinear} shows dramatic savings in iterations for higher levels. Furthermore, the reason for the negative savings for a level $L=2$ stochastic approximation is that the interpolant is not yet accurate enough to overcome the additional cost of the acceleration. 

\begin{figure}
\begin{center}
	\includegraphics[width=7.5cm]{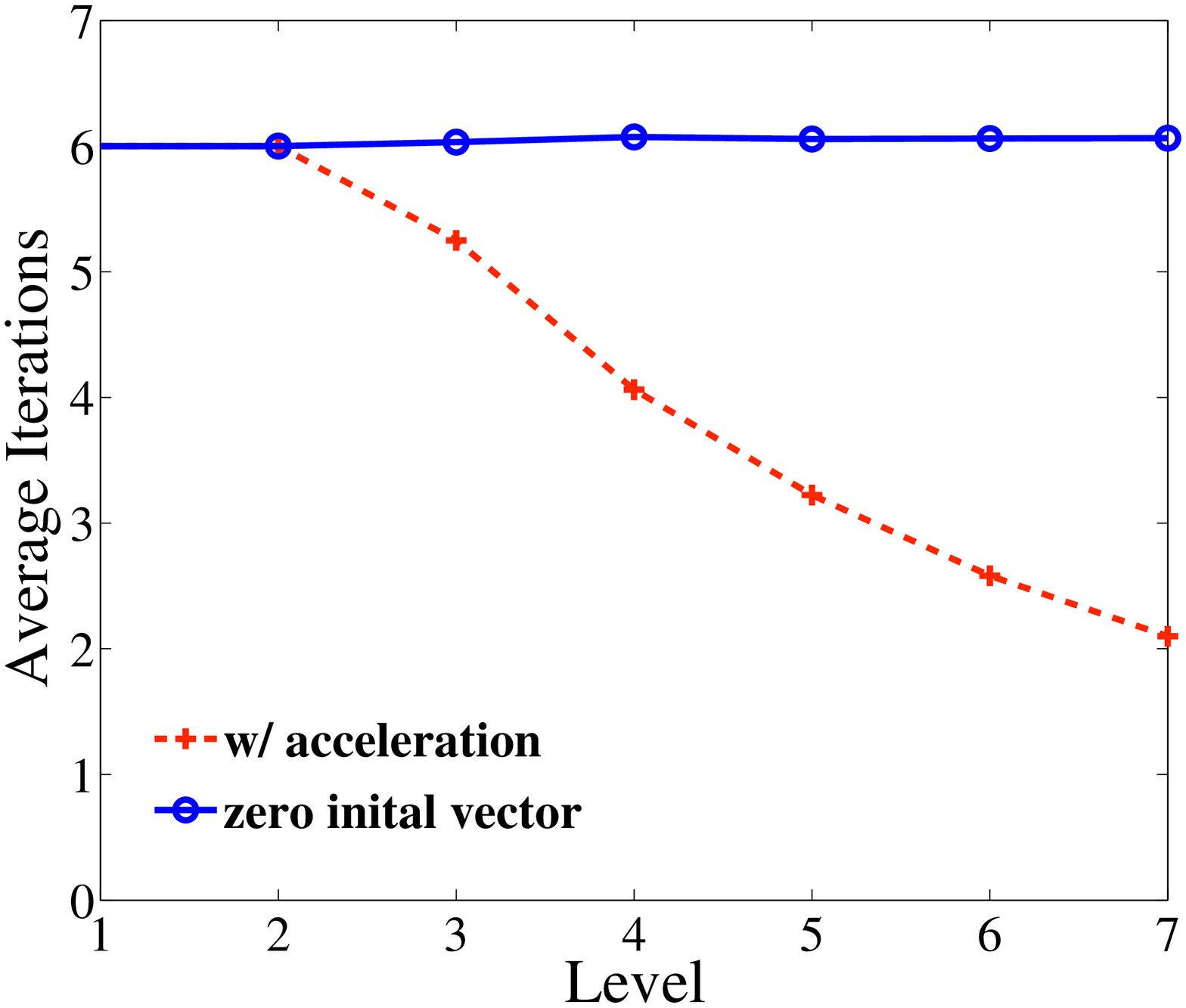}
	\includegraphics[width=7.5cm]{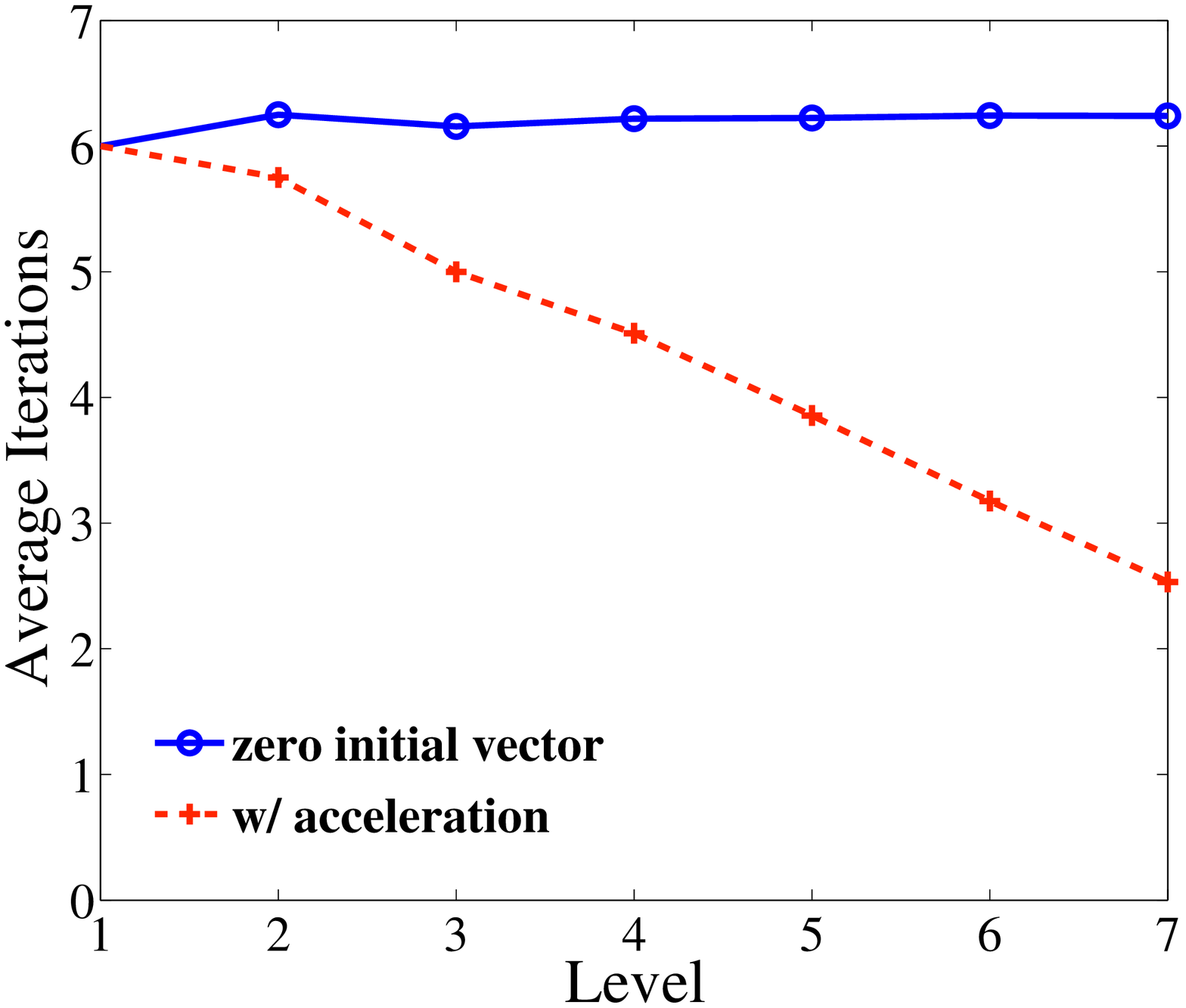}
\end{center}
\caption{ Average number of Newton iterations per level for solving problem \eqref{prob:nonlinear} with $F[u] = uu'$ (left) and $F[u] = u^5$ (right).}
\label{fig:nonlinear}
\end{figure}

\begin{table}
\begin{center}
\begin{tabular}{|c || c | c | c | c | c | }
\hline
SC Level	&  2	& 3 & 4 & 	5 	& 6 \\ \hline \hline
 $F[u] = u^5$, acc 	&	.03018	& .113832	&	.2746  	&	.7039	&	2.33314 	\\	
 $F[u] = u^5$, zero	&	.025976	& .119256	&	.339678	&	.949184	&	2.61958	\\ \hline
{\bf \% Savings} 	&	  -16.2  	&  4.5 	&  19.2 		&   	25.8 		&   	10.9  \\ \hline
 $F[u] = uu'$, acc 	&	.027754	& .089082	&	.22706	&	.629451	&	2.05741 \\ 
 $F[u] = uu'$, zero 	&	.026527	& .090435 &	.273355	&	.895027	&	2.4008 \\\hline
 {\bf \% Savings} 	&	 -4.6 		&   1.5   	& 16.9  		& 	29.7  	& 	14.3   \\ \hline
 \end{tabular}
\end{center}
\caption{Computational time in seconds for computing solution to problem \ref{prob:nonlinear}.}
 \label{tbl:nonlineartime}
\end{table}

\section{Conclusion}\label{sec:conc}

In this work, we proposed and analyzed an acceleration method for construction of sparse interpolation-based approximate solutions to PDEs with random input parameters. The acceleration method exploits the sequence of increasingly accurate approximate solutions to provide increasingly good initial guesses for the underlying iterative solvers that are used at each sample point. We have developed this method using a global Lagrange polynomial basis but the method can easily be extended to other non-intrusive methods.

While our method takes advantage of the natural structure provided by hierarchical SC methods, we do not take advantage of any hierarchy in the spatial approximation. As explained in Remark \ref{rem:ML}, our method may be used in combination with the multilevel method to accelerate the construction of stochastic operators, and reuse information from level to level. The combination of the acceleration scheme with multilevel methods will be the subject of future work.

 We rigorously studied error estimates in the special the case of linear elliptic PDEs with random inputs, providing complexity estimates for the proposed method. Several numerical examples confirm the expected performance.
 While the analysis of \S\ref{sec:estimate} applies to linear stochastic PDEs, the acceleration method may be even more well suited to nonlinear problems, as convergence rates may be improved, based on the choice of a good initial guess for nonlinear iterative solvers. A final numerical example demonstrates the advantage of our approach to nonlinear problems. A more rigorous study of acceleration for nonlinear solvers and extension to time dependent problems may provide interesting opportunities in the future.

\bibliographystyle{siam}
\bibliography{aSCrefs}

\end{document}